\documentclass[11pt,a4paper]{article}

\usepackage{styling}
\usepackage[utf8]{inputenc}

\hypersetup{
  pdftitle={An Improved Drift Theorem for Balanced allocations},
  pdfauthor={Dimitrios Los, Thomas Sauerwald}
  }

\allowdisplaybreaks

\newcommand{\COne}{{\hyperlink{c1}{\ensuremath{\mathcal{C}_1}}}\xspace}
\newcommand{\CTwo}{{\hyperlink{c2}{\ensuremath{\mathcal{C}_2}}}\xspace}
\newcommand{\DZero}{{\hyperlink{d0}{\ensuremath{\mathcal{D}_0}}}\xspace}
\newcommand{\DOne}{{\hyperlink{d1}{\ensuremath{\mathcal{D}_1}}}\xspace}

\begin{document}

\title{An Improved Drift Theorem for Balanced Allocations\footnote{Some of the results of this paper were presented at SPAA 2022~\cite{LS22Batched}.}}

\author[]{Dimitrios Los\thanks{\texttt{dimitrios.los@cl.cam.ac.uk}} }
\author[]{Thomas Sauerwald\thanks{\texttt{thomas.sauerwald@cl.cam.ac.uk}}}
\affil[]{Department of Computer Science \& Technology, University of Cambridge}

\date{}

\maketitle

\begin{abstract}
In the balanced allocations framework, there are $m$ jobs (balls) to be allocated to $n$ servers (bins). The goal is to minimize the \textit{gap}, the difference between the maximum and the average load. 

Peres, Talwar and Wieder~(RSA 2015) used the hyperbolic cosine potential function to analyze a large family of allocation processes including the $(1+\beta)$-process and graphical balanced allocations. The key ingredient 
was to prove that the potential drops 
in every step, i.e., 
\emph{a drift inequality}. 

In this work we improve the drift inequality so that $(i)$ it is asymptotically tighter, $(ii)$~it assumes weaker preconditions, $(iii)$~it applies not only to processes allocating to more than one bin in a single step and $(iv)$~to processes allocating a varying number of balls depending on the sampled bin. 
Our applications include the processes of (RSA 2015), but also several new processes, and we believe that our techniques may lead to further results in future work.

	\medskip


\end{abstract}

\section{Introduction}

We study the classical problem of allocating $m$ balls (jobs) into $n$ bins (servers). This  framework also known as \textit{balls-into-bins} or \textit{balanced allocations}~\cite{ABKU99}~is a popular abstraction for various resource allocation and storage problems such as load balancing, scheduling or hashing (see surveys~\cite{MRS01,W17}).

For the simplest allocation process, called \OneChoice, each 
of the $m$ balls is allocated to a bin chosen independently and uniformly at random. It is well-known that the maximum load is $\Theta( \log n / \log \log n)$ \Whp\footnote{In general, with high probability refers to probability of at least $1 - n^{-c}$ for some constant $c > 0$.}~for $m=n$, and $m/n + \Theta( \sqrt{ (m/n) \cdot \log n})$ \Whp~for $m \gg n$. 

Azar, Broder, Karlin and Upfal~\cite{ABKU99} (and implicitly Karp, Luby and Meyer auf der Heide~\cite{KLM96}) proved that each ball is allocated to the lesser loaded of $d \geq 2$ bins chosen uniformly at random, then the \textit{gap} between the maximum and average load drops to $\log_d \log n + \Oh(1)$ \Whp, if $m=n$. Berenbrink, Czumaj, Steger and V\"ocking~\cite{BCSV06} proved that the same bound also applies for arbitrary $m \geq n$. This dramatic improvement from $d=1$ (\OneChoice) to $d=2$ (\TwoChoice) is known as ``power of two choices''.

Several variants of the \TwoChoice process have been studied making different trade-offs between number of samples and guarantee on the gap. Mitzenmacher~\cite{M96} introduced the \OnePlusBeta-process, which in each step performs \TwoChoice with probability $\beta$ and \OneChoice otherwise. This process is more \textit{sample efficient} than \TwoChoice as it samples $(1+\beta)$ bins in expectation and was recently shown to have an asymptotically better gap than \TwoChoice in settings with \textit{outdated information}~\cite{LS23Batched}. Peres, Talwar and Wieder~\cite{PTW15} proved an $\Oh(\log n/\beta + \log(1/\beta)/\beta)$ bound on the gap, where for $\beta = o(1/\poly(n))$ the second term dominates; and also an $\Omega(\log n/\beta)$ lower bound for any $\beta$ bounded away from $1$. The \OnePlusBeta-process has been used to analyze \textit{population protocols}~\cite{AAG18,AGR21}, \textit{distributed data structures}~\cite{ABKLN18,AK0N17} and \textit{online carpooling}~\cite{GK0020}.

Another sample efficient variant of \TwoChoice is the family of \TwoThinning processes~\cite{FG21,FGL21}, where two bins are sampled uniformly at random and the ball is allocated in an online fashion. In \cite{FGL21}, a lower bound of $\Omega(\log n/\log \log n)$ on the gap was shown for any \TwoThinning process  and \cite{FGL21} designed an adaptive process that achieves this.  In~\cite{LS22Queries}, the $\Quantile(\delta)$ process was studied which is a special instance of \TwoThinning where the first sampled bin is accepted only if it is within the $(1-\delta) \cdot n$ lightest bins. Extensions of \TwoThinning making use of $d \geq 2$ samples in each step have been studied in~\cite{LS22Queries,BKSS13,CS01,FL20}.

Several of these processes have been analyzed in the \Weighted setting~\cite{TW07,PTW15}, where balls have \textit{weights} sampled from a distribution with a finite moment generating function. In particular, the upper bounds for the \OnePlusBeta-process still hold in the presence of weights.

In~\cite{LSS21}, the \Twinning process was studied, where a single bin $i \in [n]$ is sampled in each step $t \geq 0$ and \textit{two} balls are allocated if $x_i^t < t/n$ and \textit{one} ball otherwise. This process is more sample efficient than \OneChoice, but also achieves \Whp~a $\Theta(\log n)$ gap.

Mitzenmacher, Prabhakar and Shah~\cite{MPS02} studied a balanced allocation process which takes one bin sample in each step and in addition can maintain one bin in a \textit{cache}. In~\cite{LSS23}, a relaxed version of this process, called \ResetMemory was studied where the cache is \textit{reset} once every $\Oh(1)$ steps.

In~\cite{PTW15}, the analysis works for a wide range of processes defined through a \textit{probability allocation vector} $p^t$, where $p_i^t$ gives the probability to allocate to the $i$-th heaviest bin in step $t \geq 0$. Consequently, the analysis not only applies to the $(1+\beta)$-process, but through a majorization argument it also applies to \textit{graphical balanced allocation}. In this setting,  the bins are vertices of a graph $G = ([n], E)$ and in each step one edge is sampled uniformly at random and the ball is allocated to the least loaded of the two adjacent bins. For graphs with conductance $\phi$, they prove an $\Oh(\log n/\phi)$ bound on the gap. However, due to the involved majorization argument, the obtained gap bounds do not apply for weights. %

In the \Batched setting~\cite{BCEFN12,BFKMNW18}, the loads of the bins are \textit{periodically updated} every $b$ steps, capturing a setting of outdated information. Berenbrink, Czumaj, Englert, Friedetzky and Nagel~\cite{BCEFN12}, proved an $\Oh(\log n)$ upper bound on the gap for \TwoChoice when $b = n$. This has been recently refined for \TwoChoice and generalized for various processes and $b \geq n$ in~\cite{LS22Batched,LS23Batched}. In the \textit{parallel setting}~\cite{LPY19,LW11,GM11,S96,ACMR98}, allocations are made after a few rounds of communications between balls and bins.

\subsection{Our Results}

In this work we refine and generalize the drift theorem in~\cite{PTW15} (as detailed in \cref{sec:refined_drift_thm}) and provide a self-contained proof in \cref{sec:drift_thm_proof}. Then, we apply it to obtain the following results:
\begin{itemize}\itemsep0pt
    \item For the \OnePlusBeta-process, we prove an upper bound on the gap of $\Oh(\log n/\beta)$ (\cref{thm:one_plus_beta_upper_bound}).
    This is tight for any $\beta \in (0, 1)$ that is bounded away from $1$ (\cref{rem:one_plus_beta_lower_bound}). The same upper bound on the gap holds also in the \Weighted setting, which for constant $\beta \in (0,1)$ matches a general lower bound of $\Omega(\log n)$ if the weights are sampled from an exponential distribution with constant mean (\cref{rem:weighted_lower_bound}).
    \item For graphical balanced allocation on graphs with conductance $\phi$, we extend the $\Oh(\log n/\phi)$ bound on the gap to the weighted case (\cref{thm:weighted_graphical}), making progress on \cite[Open Problem 1]{PTW15}, which states:
\begin{quote}
\textit{Graphical processes in the weighted case. The analysis of section 3.1 goes through the majorization approach and therefore applies only to the unweighted case. It would be interesting to analyze such processes in the weighted case as well.}
\end{quote}
\item For the $\Quantile(\delta)$ process we prove an~$\Oh(\log n/\delta)$ bound on the gap for any $\delta \leq \frac{1}{2}$ and  $\Oh(\log n/(1-\delta))$ for any $\delta > \frac{1}{2}$ (\cref{thm:quantile_upper_bound}), including non-constant $\delta$. %
    \item We analyze the \TwinningWithQuantile process, a variant of \Twinning~\cite{LSS21} based on quantiles, and we prove an $\Oh(\log n)$ bound on the gap (\cref{thm:twinning_with_quantile_upper_bound}), which by \cref{cor:twinning_with_quantile_lower_bound} is tight. 
    \item We also introduce \QuantileWithPenalty, a variant of the \TwoThinning process with \textit{penalties}, where we allocate one ball to the first sampled bin and two balls to the second sampled bin. Somewhat surprisingly, despite penalizing the second sample, we also prove \Whp~an $\Oh(\log n)$ bound on the gap for an instance of this process with quantiles (\cref{thm:twinning_with_penalty_upper_bound}), which by \cref{cor:twinning_with_penalty_lower_bound} is tight and within a $\Theta(\log \log n)$ factor from the bound in the normal \TwoThinning setting.
    \item For \ResetMemory, we prove an~$\Oh(\log n)$ bound on the gap, which we also show to be tight in \cref{cor:reset_memory_lower_bound}. This upper bound also applies in the \Weighted setting. 
    \item Next, we prove an $\Oh((b/n) \cdot \log n)$ bound on the gap for a large family of processes in the \Batched setting with $b \geq n$ and unit weights. This is a subset of the results in \cite{LS22Batched,LS23Batched} and it is included to demonstrate the applicability of the improved drift inequality.
    \item Finally, the new drift theorem allows us to deduce that the expectation of the hyperbolic cosine potential is $\Oh(n)$ for a large family of processes. This allows us to prove bounds on the number of bins above a certain load, which is useful in the application of layered induction~\cite{ABKU99,LS22Batched,LS23Batched,LS22Queries}. In \cref{lem:lower_bound_overload_height,lem:lower_bound_underload_height}, we show that these are in some sense tight for a large family of processes.
\end{itemize}

\paragraph{Organization.} This work is structured as follows. In \cref{sec:notation_and_processes}, we introduce the notation for balanced allocation processes and define formally the various processes and settings used in this work. In \cref{sec:refined_drift_thm}, we present the improved drift theorem and compare it to that in~\cite{PTW15}. In \cref{sec:drift_thm_proof}, we prove the drift theorem. Next, in \cref{sec:three_useful_lemmas} we prove some auxiliary lemmas that help verify the preconditions of the drift inequality. In \cref{sec:applications}, we present the application of the drift inequality to analyze the \OnePlusBeta-process, graphical balanced allocation, \Quantile,  \TwinningWithQuantile, \QuantileWithPenalty, \ResetMemory and the \Batched setting for a wide family of processes. In \cref{sec:lower_bounds}, we prove lower bounds for the aforementioned processes. Finally, in \cref{sec:conclusion}, we conclude with some open problems.

\section{Notation, Settings and Processes} \label{sec:notation_and_processes}

\subsection{Notation}

We consider processes that allocate balls into $n$ bins, labeled $[n] := \{ 1, \ldots, n \}$.
The \textit{load vector} at step $t$, that is, after $t$ allocations, is $x^{t}=(x_1^{t},x_2^{t},\ldots,x_n^{t})$ starting with $x_i^{0} = 0$ for all $i \in [n]$.
Also $y^{t}=(y_1^{t},y_2^{t},\ldots,y_n^{t})$ will be the permuted load vector, sorted non-increasingly in load. In this load vector, each bin has a \textit{rank} $\Rank^{t}(i)$, where $\Rank^{t}$ forms a permutation of $[n]$ and satisfies
$
   y_{\Rank^{t}(i)}^{t}=x_i^{t}.
$
Following previous work, we analyze allocation processes in terms of the 
\[
\Gap(t):= \max_{i \in [n]} x_i^{t} - \frac{t}{n} = y_1^{t},
\]
i.e., the difference between maximum and average load at time $t \geq 0$. In all of our bounds, we also bound the difference between the maximum and minimum load. We also define the \textit{height} of a ball as $h\geq 1$ if it is the $h$-th ball added to the bin. Next we define $\mathfrak{F}^{t}$, the \textit{filtration} corresponding to the first $t$ allocations of the process (so in particular, $\mathfrak{F}^{t}$ reveals $x^{t}$). A \textit{probability vector} $p \in \R^n$ is any vector satisfying $\sum_{i = 1}^n p_i = 1$ and $p_i \in [0,1]$ for all $i \in [n]$. Any \textit{allocation process} is defined by a \emph{probability allocation vector} $p^t := p^t(\mathfrak{F}^t)$, which is the probability vector where $p_i^t$ is probability of allocating a ball to the $i$-th heaviest bin in step $t$. For some of our analysis, it will be convenient to merge multiple steps into \textit{rounds}.

For two probability allocation vectors $p$ and $q$ (or analogously, for two sorted load vectors), we say that $p$ \textit{majorizes} $q$ if 
$
 \sum_{i=1}^{k} p_i \geq \sum_{i=1}^k q_i,
$ for all $1 \leq k \leq n$. We denote this by $p\succeq q$.

\begin{thm}[Majorization, cf.~{\cite[Theorem 3.5]{ABKU99}}] \label{thm:majorization}
Consider any two processes $\mathcal{P}$ and $\mathcal{Q}$ with probability allocation vectors $p^t$ and $q^t$ respectively, such that $p^t \succeq q^t$ for every step $t \geq 0$. Then, $y_{\mathcal{P}}^t \succeq y_{\mathcal{Q}}^t$ and in particular $(y_{\mathcal{P}}^t)_1 \geq (y_{\mathcal{Q}}^t)_1$ and $(y_{\mathcal{P}}^t)_n \leq (y_{\mathcal{Q}}^t)_n$.
\end{thm}

Many statements in this work hold only for sufficiently large $n$, and several constants are chosen generously with the intention of making it easier to verify some technical inequalities.

\subsection{Settings}

\paragraph{The \Weighted setting.} We will now extend the definitions to \emph{weighted balls} into bins. To this end, let $w^t \geq 0$ be the weight of the $t$-th ball to be allocated ($t \geq 1$). By $W^{t}$ we denote the total weights of all balls allocated after the first $t \geq 0$ allocations, so $W^t := \sum_{i=1}^n x_i^{t} = \sum_{s=1}^t w^s$. The \textit{normalized loads} are $\tilde{x}_i^{t} := x_i^t - \frac{W^t}{n}$, and with $y_i^t$ being again the decreasingly sorted, normalized load vector, we have $\Gap(t)=y_1^t$. 

In the $\FiniteMgf(\zeta)$ setting, the weight of each ball will be drawn independently from a fixed distribution $\mathcal{W}$ over $[0,\infty)$. Following~\cite{PTW15}, we assume that the distribution $\mathcal{W}$ satisfies:
\begin{itemize}
  \item $\ex{\mathcal{W}} = 1$.
  \item $\ex{e^{\zeta \mathcal{W}} } = \Theta(1)$ for $\zeta := \zeta(n) \in (0, 1]$.
\end{itemize}
Specific examples of distributions satisfying above conditions (after scaling) are the geometric, exponential, binomial and Poisson distributions.

Similar to the arguments in~\cite{PTW15}, the above two assumptions imply that:
\begin{lem} \label{lem:bounded_weight_moment}
Consider any $\FiniteMgf(\zeta)$ distribution for any $\zeta \in (0, 1]$. Then, there exists $S := S(\zeta) \geq 1/\zeta$, such that for any  $\gamma \in (0, \zeta/2]$ and any $\ell \in [-1,1]$,
\[
\Ex{e^{\gamma \cdot \ell \cdot \mathcal{W}}} \leq 1 + \ell \cdot \gamma + S \cdot \ell^2 \cdot \gamma^2.
\]
\end{lem}
As this parameter $S$ is used in most of the upper bounds involving the $\FiniteMgf(\zeta)$ weights, we often refer to the setting as $\FiniteMgf(S)$.
\begin{proof}
This proof closely follows the argument in~\cite[Lemma 2.1]{PTW15}. Let $M(z) = \Ex{ e^{z \mathcal{W}}}$, then using Taylor's Theorem (mean value form remainder),  for any $z \in [-\gamma,\gamma]$ there exists $\xi \in [-\gamma, \gamma]$ such that
\[
M(z) = M(0) + M'(0) \cdot z + M''(\xi) \cdot \frac{1}{2} \cdot z^2 = 1 + z + M''(\xi) \cdot \frac{1}{2} \cdot z^2.
\]
By the assumptions on $\gamma$ and $\zeta$,
\begin{align*}
M''(\xi) &= \ex{\mathcal{W}^2 e^{\xi \mathcal{W}}} \\ &\stackrel{(a)}{\leq} \sqrt{\ex{\mathcal{W}^4 } \cdot \ex{e^{2\xi \mathcal{W}} }} \\ &\stackrel{(b)}{\leq} \frac{1}{2} \cdot \left(
 \ex{\mathcal{W}^4} + \ex{e^{2\xi \mathcal{W}} }
\right) \\
& \stackrel{(c)}{\leq} \frac{1}{2} \cdot 
\left(
 \left( \frac{8}{\zeta} \cdot \log\left( \frac{8}{\zeta} \right) \right)^{4} + \ex{ e^{\zeta \mathcal{W}}} + \ex{ e^{2\zeta \mathcal{W}}}
\right),
\end{align*}
where $(a)$ uses the Cauchy-Schwartz inequality $| \Ex{ X \cdot Y}| \leq \sqrt{\Ex{X^2} \Ex{Y^2} }$ for random variables $X$ and $Y$, $(b)$ uses the AM-GM inequality, and $(c)$ uses \cref{lem:s_bound}. Now defining 
\[
S:=\max\left\lbrace  \left(\frac{8}{\zeta} \cdot \log\left( \frac{8}{\zeta} \right) \right)^{4} , \ex{ e^{\zeta \mathcal{W}}} + \ex{ e^{2\zeta \mathcal{W}}} \right\rbrace \geq 1,
\]
using that $\zeta \leq 1$ and choosing $z:=\ell \cdot \gamma$,
the lemma follows.
\end{proof}

For some of the processes we consider also a \Weighted setting where the weight distribution depends on the bin chosen for allocation (cf.~the \TwinningWithQuantile and \QuantileWithPenalty processes).

\paragraph{The \Batched setting.} In the \Batched setting, the balls are allocated in \textit{batches} of size $b \in \N_+$ using the probability allocation vector at the beginning of the batch. In this work, we only consider the \Batched setting with unit weight balls.

\begin{framed}
\vspace{-.45em} \noindent
\underline{\Batched Setting}: \\
\textsf{Parameters:} Batch size $b \in \N_+$, probability allocation vector $p^t$.
\\
\textsf{Iteration:} At step $t = 0, b, 2b, \ldots$:
\begin{enumerate}\itemsep0pt
    \item Sample $b$ bins $i_1,i_2,\ldots,i_b$ independently from $[n]$ following $p^t$.
    \item Update 
    \[
    \tilde{y}_{i}^{t+b} := y_{i}^{t} + \sum_{j=1}^b \mathbf{1}_{i_j=i} - \frac{b}{n}, \quad \text{for each $i \in [n]$}.
    \]
    \item Let $y^{t+b}$ be the vector $\tilde{y}^{t+b}$, sorted non-increasingly.
\end{enumerate}
\vspace{-0.4cm}
\end{framed}

\subsection{Processes}

In this part, we include the formal description of all allocation processes considered in this work. Recall that $w^{t+1}$ is the weight of the $(t+1)$-th ball to be allocated at step $t \geq 0$.

\begin{samepage}
\begin{framed}
\vspace{-.45em} \noindent
\underline{\OneChoice Process:} \\
\textsf{Iteration:} At step $t \geq 0$, sample one bin $i \in [n]$ independently and uniformly at random. Then, update:  
    \begin{equation*}
     x_{i}^{t+1} = x_{i}^{t} + w^{t+1}.
 \end{equation*}\vspace{-1.5em}
\end{framed}
\end{samepage}
\noindent We continue with a formal description of the \TwoChoice process.
\begin{samepage}
\begin{framed}
\vspace{-.45em} \noindent
\underline{\TwoChoice Process:} \\
\textsf{Iteration:} At step $t \geq 0$, sample two bins $i_1, i_2 \in [n]$, independently and uniformly at random. Let $i \in \{i_1, i_2 \}$ be such that $x_{i}^{t} = \min\{ x_{i_1}^t,x_{i_2}^t\}$, favoring bins with a higher index. Then, update:  
    \begin{equation*}
     x_{i}^{t+1} = x_{i}^{t} + w^{t+1}.
 \end{equation*}\vspace{-1.5em}
\end{framed}
\end{samepage}
It is immediate that the probability allocation vector of \TwoChoice is
\begin{equation*}
    p_{i} = \frac{2i-1}{n^2}, \qquad \mbox{ for all $i \in [n]$.}
\end{equation*}

Following~\cite{M96}, we recall the definition of the \OnePlusBeta-process which \textit{interpolates} between \OneChoice and \TwoChoice:
\begin{samepage}
\begin{framed}
\vspace{-.45em} \noindent
\underline{($1+\beta$)-Process:}\\
\textsf{Parameter:} A \textit{mixing factor} $\beta \in (0,1]$.\\
\textsf{Iteration:} At step $t \geq 0$, sample two bins $i_1, i_2 \in [n]$, independently and uniformly at random. Let $i \in \{ i_1, i_2 \}$ be such that $x_{i}^{t} = \min\big\{ x_{i_1}^t,x_{i_2}^t \big\}$, favoring bins with a higher index. Then, update:  
    \begin{equation*}
    \begin{cases}
     x_{i}^{t+1} = x_{i}^{t} + w^{t+1} & \mbox{with probability $\beta$}, \\
      x_{i_1}^{t+1} = x_{i_1}^{t} + w^{t+1} & \mbox{otherwise}.
   \end{cases}
 \end{equation*}\vspace{-1.em}
\end{framed}
\end{samepage}

The probability allocation vector of the \OnePlusBeta-process is given by:
\begin{equation*}
    p_{i} =
    (1-\beta) \cdot \frac{1}{n} + \beta \cdot \frac{2i-1}{n^2}, \qquad \mbox{ for all $i \in [n]$.}
\end{equation*}

We now define a variant of the \TwoChoice process with incomplete information, called the \Quantile process, which is also an instance of \TwoThinning (see \cref{fig:quantile_process}).

\begin{figure}[H]
    \centering
    \includegraphics[scale=0.45]{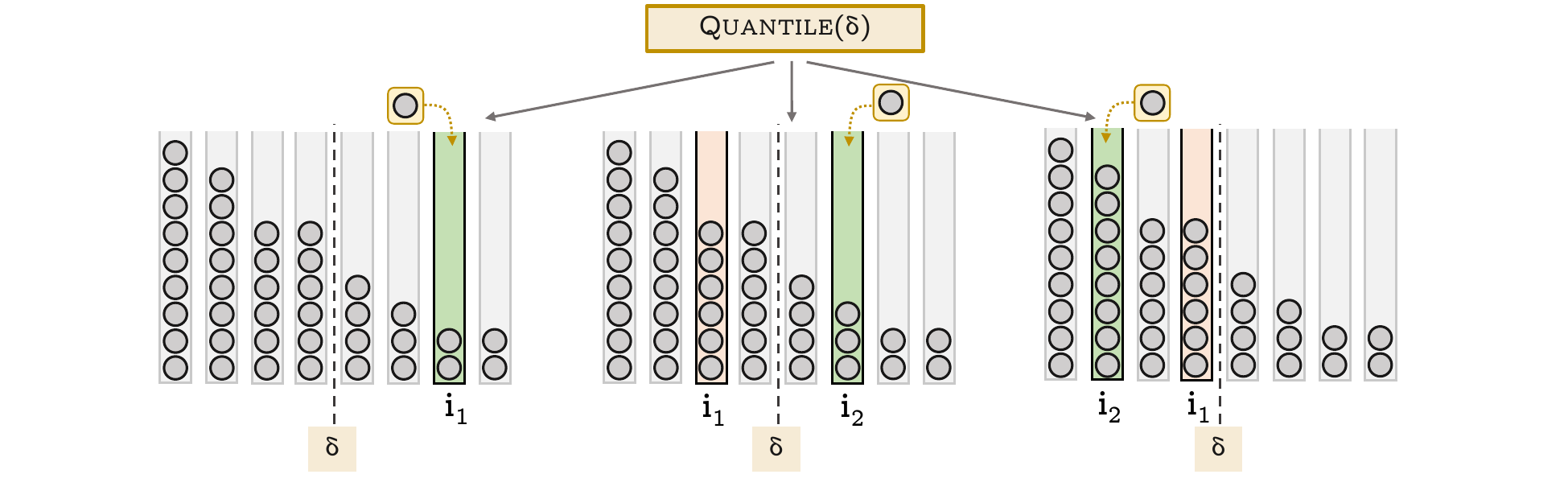}
    \caption{The two cases for the $\Quantile(\delta)$ process: (\textbf{left}) Allocate to the first sample if its rank is above $n \cdot \delta$, (\textbf{center}) allocate to the second sample (in this case a light bin) if the rank of the first bin was at most $n \cdot \delta$ and (\textbf{right}) allocate to the second sample which could also be a heavy bin.}
    \label{fig:quantile_process}
\end{figure}
\begin{framed}
	\vspace{-.45em} \noindent
	\underline{$\Quantile(\delta)$~Process:}\\
	\textsf{Iteration:} At step $t \geq 0$, sample two bins $i_1, i_2 \in [n]$ independently and uniformly at random. Then, update:
	\begin{equation*} 
		\begin{cases}
			x_{i_1}^{t+1} = x_{i_1}^{t} + w^{t+1} & \mbox{if $\Rank^t(i_1) > n \cdot \delta$}, \\
			x_{i_2}^{t+1} = x_{i_2}^{t} + w^{t+1}  & \mbox{otherwise}.
		\end{cases}
	\end{equation*}
 \vspace{-0.4cm}
\end{framed}

The probability allocation vector for the $\Quantile(\delta)$ process is given by:
\begin{align} \label{eq:quantile_prob_vector}
p_i = \begin{cases}
\frac{\delta}{n} & \text{if }i \leq n \cdot \delta, \\
\frac{1 + \delta}{n} & \text{otherwise}.
\end{cases}
\end{align}

The \Twinning process was introduced and analyzed in~\cite{LSS21}, and works by sampling one bin $i \in [n]$ uniformly at random and allocating \textit{two} balls if $x_i^t \leq \frac{t}{n}$, otherwise allocating \textit{one} ball. Here, we consider a variant of this process, so that the decision function is $\Rank^t(i) > n \cdot \delta$ instead of $x_i^t < \frac{t}{n}$ (see \cref{fig:twinning_with_quantile}).

\begin{framed}
	\vspace{-.45em} \noindent
	\underline{$\TwinningWithQuantile(\delta)$~Process:}\\
	\textsf{Iteration:} At step $t \geq 0$, sample a bin $i \in [n]$ independently and uniformly at random. Then, update:
	\begin{equation*} x_{i}^{t+1} =
		\begin{cases}
			x_{i}^{t} + 2 & \mbox{if $\Rank^t(i) > n \cdot \delta$}, \\
			x_{i}^{t} + 1  & \mbox{otherwise}.
		\end{cases}
	\end{equation*}
 \vspace{-0.4cm}
\end{framed}

Note that this process allocates $1 \cdot \delta^2 + 2 \cdot (1 - \delta^2) = 2 - \delta^2$ balls per sample in expectation, so it is more sample efficient than \OneChoice.\footnote{In comparison to the original \Twinning process~\cite{LSS21}, it is sample-efficient in every step (not only when the quantile of the average load has stabilized), but it requires knowledge of the ordering of the bins.} In \cref{sec:twinning_with_quantile}, we show that this process also has an $\Oh(\log n)$ gap.

\begin{figure}[H]
    \centering
    \includegraphics[scale=0.45]{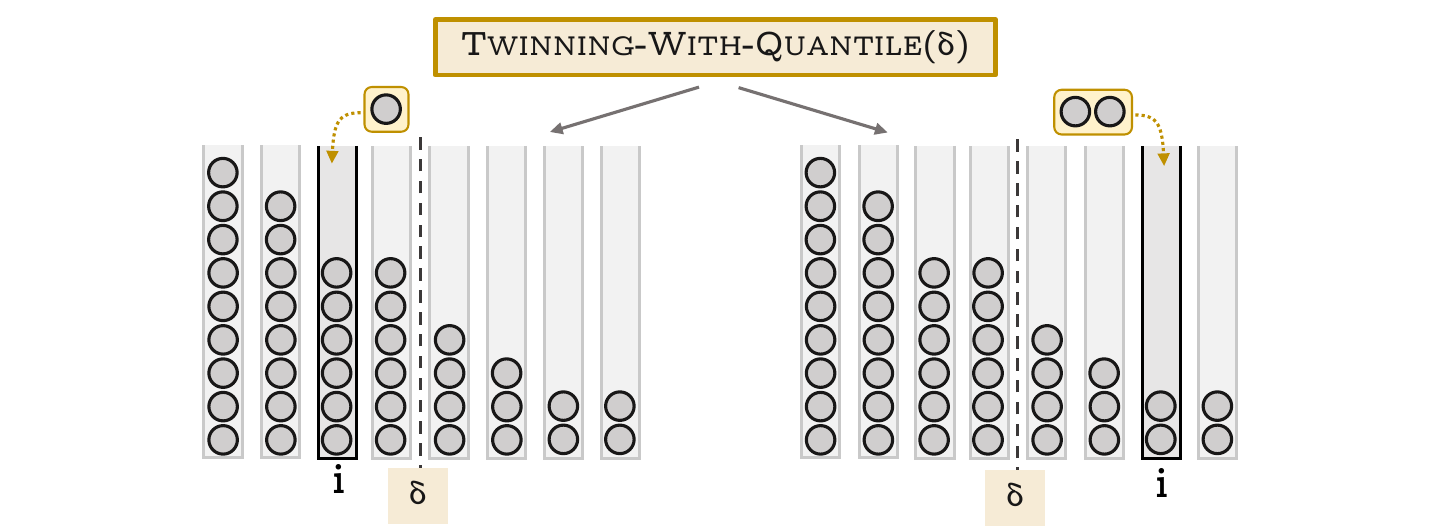}
    \caption{The two cases for the $\TwinningWithQuantile(\delta)$ process: (\textbf{left}) Allocate one ball if the rank of the sampled bin is at most $n \cdot \delta$, otherwise (\textbf{right}) two balls to the sampled bin.}
    \label{fig:twinning_with_quantile}
\end{figure}

Further, we analyze the following variant of the \Quantile process, which allocates one ball to the first sample or two balls to the second sample. This can be seen as a version of the \TwoThinning process which penalizes allocations to the second sample (see \cref{fig:quantile_with_penalty_process}).

\label{sec:quantile_with_penalty_def}

\begin{framed}
	\vspace{-.45em} \noindent
	\underline{$\QuantileWithPenalty(\delta)$~Process:}\\
	\textsf{Iteration:} At step $t \geq 0$, sample two bins $i_1, i_2 \in [n]$ independently and uniformly at random. Then, update:
	\begin{equation*} 
		\begin{cases}
			x_{i_1}^{t+1} = x_{i_1}^{t} + 1 & \mbox{if $\Rank^t(i_1) > n \cdot \delta$}, \\
			x_{i_2}^{t+1} = x_{i_2}^{t} + 2  & \mbox{otherwise}.
		\end{cases}
	\end{equation*}
 \vspace{-0.4cm}
\end{framed}

\begin{figure}[H]
    \centering
    \includegraphics[scale=0.45]{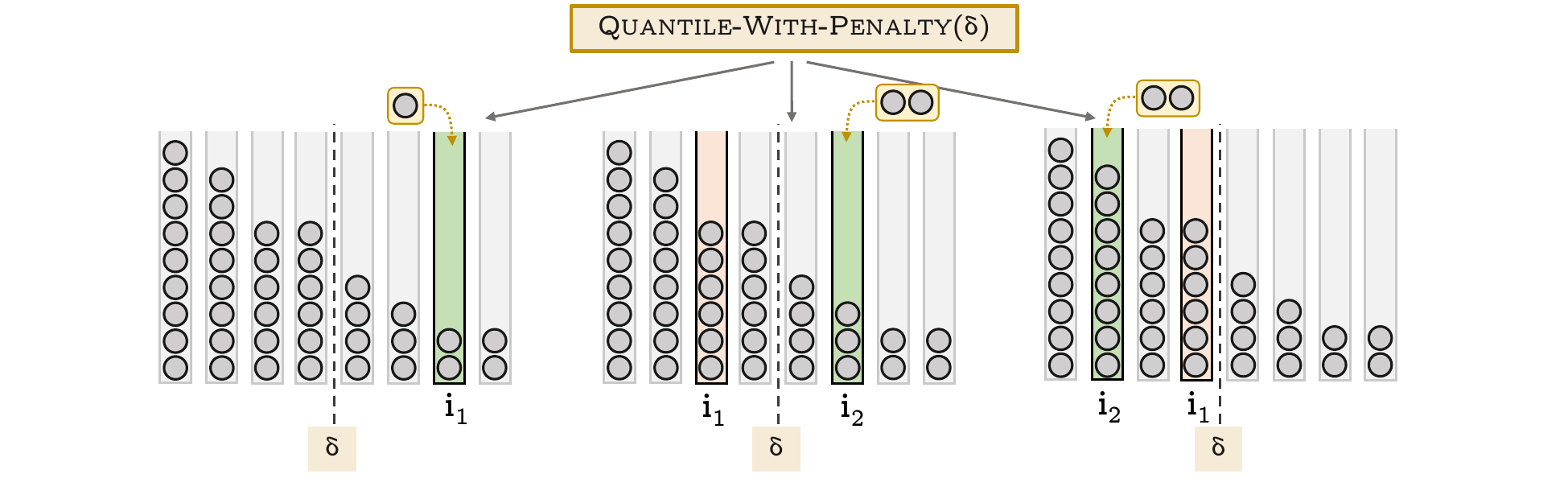}
    \caption{Three cases for the $\QuantileWithPenalty(\delta)$ process: Allocating one ball to the first sample or two balls to the second sample (cf.~\cref{fig:quantile_process}).}
    \label{fig:quantile_with_penalty_process}
\end{figure}

In~\cite{LSS23}, the following relaxation of the \Memory process was studied where the memory (or cache) is \textit{reset} every second step. As remarked on~\cite[page 6]{LSS23}, this can be seen as a sample-efficient version of the \OnePlusBeta-process.

\begin{framed}
	\vspace{-.45em} \noindent
	\underline{$\ResetMemory$~Process:}\\
	\textsf{Iteration:} At step $2t \geq 0$, sample a bin $i_1 \in [n]$ independently and uniformly at random. Then, update:
	\begin{equation*} 
		x_{i_1}^{2t+1} = x_{i_1}^{2t} + w^{2t+1}.
	\end{equation*}
    At step $2t+1\geq 0$, sample a bin $i_2 \in [n]$ uniformly at random and let $i \in \{i_1, i_2\}$ be such that $x_{i}^{2t} = \min\big\{x_{i_1}^{2t}, x_{i_2}^{2t}\big\}$, favoring bins with a higher index. Then, update:\[
        x_{i}^{2t+2} = x_{i}^{2t+1} + w^{2t+2}.
    \]
 \vspace{-0.9cm}
\end{framed}

Finally, in \textit{graphical balanced allocation} \cite{KP06,PTW15,BF21,ANS20}, we are given an undirected graph $G = ([n], E)$ with $n$ vertices corresponding to $n$ bins. For each ball to be allocated, we sample an edge $\{u,v\} \in E$ uniformly at random, and allocate the ball to the lesser loaded bin among $\{u,v\}$. Note that by taking $G$ as a complete graph, we recover the \TwoChoice process. The graphical balanced allocation setting has also been generalized further to \textit{hypergraphs}~\cite{PTW15,GMP20}.

\begin{framed}
\vspace{-.45em} \noindent
\underline{$\Graphical(G)$}\\
\textsf{Parameter:} An undirected, connected \textit{graph} $G = ([n], E)$.\\
\textsf{Iteration:} For each step $t \geq 0$, sample an edge $e=\{i_1,i_2\} \in E$ independently and uniformly at random. Let $i \in \{i_1,i_2 \}$ be such that $x_{i}^{t} = \min\{ x_{i_1}^t,x_{i_2}^t\}$, favoring bins with a higher index. Then update:  
    \begin{equation*}
     x_{i}^{t+1} = x_{i}^{t} + w^{t+1}.
 \end{equation*}\vspace{-1.5em}
\end{framed}

\paragraph{Conditions on Probability Vectors.} The drift theorem in~\cite{PTW15} applies to any process for which there exist $\eps \in (0, 1/4)$ and quantile $\delta \in \{ 1/n, \ldots, 1\}$ such that for every step $t \geq 0$, the probability allocation vector $p^t$ is non-decreasing and
\begin{align} \label{eq:ptw_p_i_conditions}
p_i^t \leq \frac{1 - 4\eps}{n}, \quad\text{for any }i \leq \frac{n}{3},\quad
\text{ and }
\quad
p_i^t \geq \frac{1 + 4\eps}{n}, \quad\text{ for any }i \geq \frac{2n}{3}.
\end{align}

As we will show now, the second part of the condition is not necessary, in the sense that a lower bound on the probability to allocate to a light bin is implied by the first condition and monotonicity of the probability allocation vector. To this end, we define the following conditions on a probability vector:
\begin{itemize}\itemsep0pt
  \item \textbf{Condition \hypertarget{d0}{$\mathcal{D}_0$}}: $(p_i)_{i \in [n]}$ is a \textit{non-decreasing} probability vector in $1 \leq i \leq n$.
  \item \textbf{Condition \hypertarget{d1}{$\mathcal{D}_1$}}: There exist constant $\delta \in (0, 1)$ and (not necessarily constant) $\eps \in (0, 1)$, such that
    \[
    p_{n \cdot \delta} \leq \frac{1 - \epsilon}{n}. 
    \]
\end{itemize}

In our refined version of the drift theorem (\cref{thm:hyperbolic_cosine_expectation}), we define the following generalized conditions \COne and $\CTwo$ for a probability vector $p$, and in \cref{pro:d0_and_d1_imply_c0_and_c1} show that \DZero and \DOne imply condition \COne:
\begin{itemize}\itemsep0pt
  \item \textbf{Condition \hypertarget{c1}{$\mathcal{C}_1$}}: There exist constant\footnote{Here by constant $\delta := \delta(n) \in \{ \frac{1}{n}, \frac{2}{n}, \ldots, 1 \}$, we mean that there exist constants $\delta_1, \delta_2 \in (0, 1)$ such that $\delta_1 \leq \delta \leq \delta_2$ for sufficiently large $n$.} $\delta \in (0, 1)$ and (not necessarily constant) $\eps \in (0, 1)$, such that
    \[
    \sum_{i=1}^{k} p_{i} \leq (1 - \epsilon) \cdot \frac{k}{n} \quad \text{for any $1 \leq k \leq n \cdot \delta$},
    \]
    and similarly
    \[
     \sum_{i=k}^{n} p_i \geq \left(1 + \epsilon \cdot \frac{\delta}{1-\delta} \right) \cdot \frac{n-k+1}{n} \quad \text{for any $n \cdot \delta +1 \leq k \leq n$}.
    \]
 \item \textbf{Condition \hypertarget{c2}{$\mathcal{C}_2$}}: There exists $C \geq 1$, such that $\max_{i \in [n]} p_i \leq \frac{C}{n}$. 
\end{itemize}
\noindent Note that any process taking $d$ uniform samples in each step satisfies condition \CTwo with $C := d$.

\begin{pro}\label{pro:d0_and_d1_imply_c0_and_c1}
Consider any probability vector $p$ that satisfies conditions \DZero and \DOne for some $\delta$ and $\eps$. Then, it also satisfies condition \COne for the same $\delta$ and $\eps$.
\end{pro}

\begin{proof}
Consider any probability vector $p$ that satisfies conditions \DZero and \DOne for some $\delta$ and $\eps$. Then, by condition \DZero we have that $p$ is non-decreasing and by \DOne that $p_{n \cdot \delta} \leq \frac{1-\epsilon}{n}$. Therefore, it follows that $p_i \leq \frac{1-\epsilon}{n}$ for $i \leq n \cdot \delta$ and so for any $1 \leq k \leq n \cdot \delta$, we have that the prefix sums satisfy
\[
 \sum_{i=1}^{k} p_i 
  \leq \frac{1-\epsilon}{n} \cdot k
  \leq (1-\epsilon) \cdot \frac{k}{n}.
\]
This also implies for the the suffix sum at $k = n \cdot \delta + 1$ that
\[
 \sum_{i=n \cdot \delta+1}^{n} p_i \geq 1 - (1 - \epsilon) \cdot \delta.
\]
Since $p$ is non-decreasing, the suffix sums for any $n \cdot \delta +1 \leq k \leq n$ satisfy
\[
 \sum_{i=k}^{n} p_i \geq \left(1 - (1 - \epsilon) \cdot \delta \right) \cdot \frac{n-k+1}{(1-\delta) n} = \left(1 + \epsilon \cdot \frac{\delta}{1-\delta} \right) \cdot \frac{n-k+1}{n}. \qedhere
\]
\end{proof}

Using this proposition, it is easy to verify that the probability allocation vector of the \OnePlusBeta-process, \TwoChoice and \Quantile satisfy the two conditions \COne and \CTwo.
\begin{pro}\label{pro:one_plus_beta}
For any $\beta \in (0,1]$, the probability allocation vector of the \OnePlusBeta-process satisfies condition \COne with $\delta := \frac{1}{4}$ and $\epsilon := \frac{\beta}{2}$ and condition \CTwo with $C := 2$.
\end{pro}

\begin{proof}
Recall that the probability allocation of the  \OnePlusBeta-process is actually
\[
 p_i = (1-\beta) \cdot \frac{1}{n} + \beta \cdot \frac{2i-1}{n^2}, \quad \text{for any }i \in [n].
\]
This shows that $p$ is increasing in $i \in [n]$ (condition \DZero), and thus also $\max_{i \in [n]} p_i \leq \frac{2}{n}$ (condition \CTwo). Further, we establish that condition \DOne holds for $\delta := \frac{1}{4}$ and $\eps := \frac{\beta}{2}$,
\[
 p_{n \cdot \delta} \leq (1-\beta) \cdot \frac{1}{n} + \beta \cdot \frac{1}{2n} = \left(1 - \frac{\beta}{2}\right) \cdot \frac{1}{n}.
\]
Therefore, by \cref{pro:d0_and_d1_imply_c0_and_c1},  condition \COne holds with the same $\delta$ and $\eps$.
\end{proof}

Note that for $\beta=1$, the \OnePlusBeta-process is equivalent to \TwoChoice, so the above statement also applies to \TwoChoice. Finally, since \TwoChoice satisfies \COne, by majorization also \DChoice for any $d > 2$ satisfies \COne with the same $\delta$ and $\epsilon$. Further,  \DChoice satisfies \CTwo with $C := d$ and thus:

\begin{pro} \label{pro:d_choice_satisfies_c0_and_c1}
For any $d \geq 2$, the probability allocation vector of \DChoice satisfies condition \COne with $\delta := \frac{1}{4}$, $\eps := \frac{1}{2}$ and condition \CTwo with $C := d$.
\end{pro}

\begin{pro} \label{pro:quantile_satisfies_c1_and_c2}
For any constant $\delta \in (0,1)$, the probability allocation vector of $\Quantile(\delta)$ satisfies condition \COne with $\delta$ and $\epsilon := 1-\delta$, and condition \CTwo with $C := 2$.
\end{pro}
\begin{proof}
Recall that the probability allocation vector of the $\Quantile(\delta)$ process is given by\[
p_{i} = \begin{cases}
 \frac{\delta}{n} & \mbox{ if $1 \leq i \leq n \cdot \delta$}, \\
 \frac{1+\delta}{n} & \mbox{ otherwise}.
\end{cases}
\]
Therefore, it trivially follows that condition \DZero holds, as well as condition \CTwo with $C:=2$. Further, we have $p_i \leq \frac{\delta}{n}$ for any $i \leq n \cdot \delta$, which means \DOne holds with $\epsilon := 1 - \delta$. By \cref{pro:d0_and_d1_imply_c0_and_c1}, condition \COne also holds.
\end{proof}

\paragraph{Regarding tie-breaking.} In the definitions of the \TwoChoice, \OnePlusBeta-process and \Graphical, we always favored bins with a higher index in case of a tie between the two sampled bins. This definition means that the \TwoChoice and \OnePlusBeta-process have a \textit{time-homogeneous} probability allocation vector. For \Graphical, this means it is possible to prove that the probability allocation vector satisfies condition \COne in every step (\cref{lem:expansion}). 

However, as we show in \cref{rem:tie_breaking},  the upper bounds that we obtain also apply to versions of the processes where we use random tie-breaking instead.  In particular, if $p$ is the original probability allocation vector, then the one with random tie-breaking is $\tilde{p} := \tilde{p}(x^t)$, where
\begin{equation} \label{eq:averaging_pi}
\tilde{p}_i(x^t) := \frac{1}{|\{ j \in [n] : x_j^t = x_i^t \}|} \cdot \sum_{j \in [n] : x_j^t = x_i^t} p_j, \quad \text{for all }i \in [n].
\end{equation}

\section{The Improved Drift Theorem (Theorem~\ref{thm:hyperbolic_cosine_expectation})} \label{sec:refined_drift_thm}

Peres, Talwar and Wieder~\cite{PTW15} analyzed the hyperbolic cosine potential for a large family of processes. In this section, we present a refined analysis for the expectation of the hyperbolic cosine potential which is asymptotically tight and also applies to a wider family of processes including weights and outdated information (see discussion below for details of the refinements). The \textit{hyperbolic cosine potential} $\Gamma := \Gamma(\gamma)$ with smoothing parameter $\gamma > 0$ is defined as
\begin{align}
\Gamma^t := \Phi^t + \Psi^t := \sum_{i = 1}^n e^{\gamma y_i^t} + \sum_{i = 1}^n e^{-\gamma y_i^t}, \label{eq:hyperbolic}
\end{align}
where $\Phi := \Phi(\gamma)$ is the \textit{overload exponential potential} and $\Psi := \Psi(\gamma)$ is the \textit{underload exponential potential}.
We also decompose $\Gamma^t$ by defining
\[
 \Gamma_i^t := \Phi_i^t + \Psi_i^t, \quad \text{where }\Phi_i^t := e^{\gamma y_i^t} \text{ and } \Psi_i^t := e^{-\gamma y_i^t} \quad \text{for any bin $i \in [n]$}.
\]
Further, we use the following shorthands to denote the changes in the potentials over one step $\Delta\Phi_i^{t+1} := \Phi_i^{t+1} - \Phi_i^t$, $\Delta\Psi_i^{t+1} := \Psi_i^{t+1} - \Psi_i^{t}$ and $\Delta\Gamma_i^{t+1} := \Gamma_i^{t+1} - \Gamma_i^{t}$.

The following theorem was proven in~\cite[Section 2]{PTW15}.

\begin{thm}[cf.~{\cite[Section 2]{PTW15}}] \label{thm:ptw_original}
Consider any allocation process with non-decreasing probability allocation vector $p$, satisfying for some $\eps \in (0, 1/4)$ that for every step $t \geq 0$,\[
p_i^t \leq \frac{1 - 4\eps}{n}, \quad\text{for any }i \leq \frac{n}{3},\quad
\text{ and }
\quad
p_i^t \geq \frac{1 + 4\eps}{n}, \quad\text{ for any }i \geq \frac{2n}{3}.
\vspace{-0.1cm}
\]
Further, consider the \Weighted setting with weights from a $\FiniteMgf(\zeta, S)$ distribution with $S \geq 1$. Then, for $\Gamma := \Gamma(\gamma)$ with $\gamma := \min\big\{\frac{\eps}{6S}, \frac{\zeta}{2} \big\}$, there exists $c = \poly(1/\eps)$ such that for any step $t \geq 0$,\[
\Ex{\left.\Delta\Gamma^{t+1} \,\right|\, \mathfrak{F}^t}
 \leq -\Gamma^t \cdot \frac{\gamma\eps}{4n} + c.
\]
\end{thm}
By \cref{lem:geometric_arithmetic} when a process satisfies this drift inequality it also satisfies $\Ex{\Gamma^t} \leq \frac{4c}{\gamma\eps} \cdot n$ for every step $t \geq 0$, and by Markov's inequality this implies that with probability at least $1 - n^{-1}$,\[
  \Gap(t) \leq \frac{1}{\gamma} \cdot \log \left( \frac{4c}{\gamma\eps} + \frac{2}{\gamma} \cdot \log n \right).
\]

As we shall describe shortly, our main theorem (\cref{thm:hyperbolic_cosine_expectation}) applies to a variety of processes and settings. However, in order to more precisely highlight the differences to \cref{thm:ptw_original}, we first state a corollary for processes with a probability allocation vector in the \Weighted setting with weights from a \FiniteMgf distribution. The two main differences between \cref{thm:ptw_original} and \cref{cor:like_ptw} are: $(i)$ that $p$ satisfies preconditions \COne and \CTwo, and $(ii)$ the additive term which changes from $\poly(1/\eps)$ to $\Oh(\gamma\eps)$.

\newcommand{\CorLikePTW}{

Consider any allocation process with probability allocation vector $p^t$ satisfying condition \COne for some constant $\delta \in (0, 1)$ and some $\eps \in (0, 1)$, and condition \CTwo for some $C > 1$, for every step $t \geq 0$.

Further, consider the \Weighted setting with weights from a $\FiniteMgf(S)$ distribution with $S \geq 1$. Then, there exists a constant $c > 0$, such that for $\Gamma := \Gamma(\gamma)$ with any $\gamma \in \big(0, \frac{\eps\delta}{16CS}\big]$ and for any step $t \geq 0$,\[
\Ex{\left.\Delta\Gamma^{t} \,\right|\, \mathfrak{F}^t}
 \leq -\Gamma^t \cdot \frac{\gamma\eps\delta}{8n} + c\gamma\eps \quad \text{and} \quad \Ex{\Gamma^t} \leq \frac{8c}{\delta} \cdot n.
\]
}
{\renewcommand{\thecor}{\ref{cor:like_ptw}}
	\begin{cor}[Of \cref{thm:hyperbolic_cosine_expectation} -- Restated, page~\pageref{cor:like_ptw}]
\CorLikePTW
	\end{cor} }
	\addtocounter{thm}{-1}

Now we state the main theorem, where the preconditions are expressed in terms of the expected change of the overload and underload potentials for any allocation process, which could be allocating to more than one bin in each round. Note that in the following theorem the probability vector $p$ need not be the probability allocation vector of the process being considered. For example, when we analyze  the $\TwinningWithQuantile(\delta)$ process in \cref{sec:twinning_with_quantile}, $p$ will be the probability allocation vector of $\Quantile(\delta)$ and not its probability allocation vector (which is the uniform vector). When the rounds consist of multiple steps, then this probability vector expresses some kind of ``average number'' of balls allocated to the $i$-th bin (cf.~\cref{lem:z_i_t_bounded}).

\newcommand{\MainHyperbolicCosineExpectation}{
Consider any allocation process and a probability vector $p^t$ satisfying condition \COne for some constant $\delta \in (0, 1)$ and some $\eps \in (0, 1)$ at every round $t \geq 0$. Further assume that there exist $K > 0$, $\gamma \in \big(0, \min\big\{1, \frac{\eps\delta}{8K}\big\} \big]$ and $R > 0$, such that for any round $t \geq 0$, the process satisfies for potentials $\Phi := \Phi(\gamma)$ and $\Psi := \Psi(\gamma)$ that for bins sorted in non-increasing order of their loads,
\[
\sum_{i = 1}^n \Ex{\left. \Delta\Phi_i^{t+1} \,\right|\, \mathfrak{F}^t} \leq \sum_{i = 1}^n \Phi_i^t \cdot \left(\left(p_i^t - \frac{1}{n}\right) \cdot R \cdot \gamma + K \cdot R \cdot \frac{\gamma^2}{n}\right),
\]
and,
\[
\sum_{i = 1}^n \Ex{\left.\Delta\Psi_i^{t+1} \,\right|\, \mathfrak{F}^t} \leq  \sum_{i = 1}^n \Psi_i^t \cdot \left(\left(\frac{1}{n} - p_i^t\right) \cdot R \cdot \gamma + K \cdot R \cdot \frac{\gamma^2}{n}\right).
\]
Then, there exists a constant $c := c(\delta) > 0$, such that for $\Gamma := \Gamma(\gamma)$ and any round $t \geq 0$,
\[
\Ex{\left. \Delta\Gamma^{t+1} \,\right|\, \mathfrak{F}^t} \leq - \Gamma^t \cdot R \cdot \frac{\gamma\eps\delta}{8n} + R \cdot c\gamma\eps,
\]
and
\[
\Ex{\Gamma^t} \leq \frac{8c}{\delta} \cdot n.
\]}

\begin{thm} \label{thm:hyperbolic_cosine_expectation}
\MainHyperbolicCosineExpectation
\end{thm}
\noindent This theorem is a refinement of \cref{thm:ptw_original} in the following ways:
\begin{itemize}
  \item When rounds consist of a single step and the allocation vector $q$ coincides with the probability vector $p$, we relax\footnote{The addition of condition \CTwo is not exactly a relaxation of \DZero and \DOne. 
  However, in the unit weight case, as we show in \cref{sec:condition_c2}, due to a majorisation argument (\cref{thm:majorization}), the gap bound in \cref{cor:like_ptw} also applies to processes satisfying conditions \DZero and \DOne (but not necessarily \CTwo).}
   the preconditions on $p$, requiring that it satisfies conditions \COne and \CTwo (instead of \DZero and \DOne).

  This allows us to apply the theorem for graphical balanced allocations with weights on $d$-regular graphs as $q$ satisfies \COne and \CTwo (as we shall see in \cref{lem:expansion}). Note that the majorization argument in~\cite[Section 3]{PTW15} only applies for unit weights.
  \item 
  It decouples upper bounding the expected change for the overload and underload potentials and their combination in proving an overall drop for the $\Gamma$ potential.
  
  This split allows us to apply the theorem $(i)$ for processes that allocate a different number of balls depending on the sampled bin, such that such as \TwinningWithQuantile (\cref{sec:twinning_with_quantile}), \QuantileWithPenalty (\cref{sec:quantile_with_penalty}) and also \ResetMemory (\cref{sec:reset_memory}) or $(ii)$ for processes that allocate balls to more than one bin in one round, such as the \Batched setting (\cref{sec:b_batched_setting}).
  \item We show that $\Gamma := \Gamma(\gamma)$ satisfies the following drift inequality for some \textit{constants} $c_1, c_2 > 0$,
  \[
    \Ex{\left.\Delta\Gamma^{t+1} \,\right|\, \mathfrak{F}^t}
    \leq -\Gamma^t \cdot \frac{c_1\gamma\eps}{n} + c_2\gamma\eps.
  \]
  \label{sec:bounds_on_bins_with_load_at_least}
  This allows us to deduce that for any round $t \geq 0$, it holds that $\Ex{\Gamma^t} \leq \frac{c_2}{c_1} \cdot n$ (\cref{lem:geometric_arithmetic}) and this directly implies the tight $\Oh\big(\frac{\log n}{\beta}\big)$ bound on the \OnePlusBeta-process for all $\beta \leq 1 - \xi$, for any constant $\xi > 0$ (\cref{thm:one_plus_beta_upper_bound}). Furthermore, the $\Oh(n)$ bound on the expectation of $\Gamma$ gives a bound on the number of bins with normalized load above a certain threshold, effectively characterizing the entire load vector. For several processes, these bounds have been critical in proving tighter bounds for the \Batched setting (\cite[Section 5]{LS22Batched} and \cite[Section 4]{LS23Batched}), and for the \Memory process~\cite{LSS23}.
\end{itemize}

The key lemma that we use to prove \cref{thm:hyperbolic_cosine_expectation} is a drift inequality that is agnostic of balanced allocation processes and is essentially an inequality involving only $\Phi, \Psi, \Gamma$ over an arbitrary (load) vector $x$ (with $y = x - \overline{x}$ being its normalized version) and a probability vector $p$ satisfying condition \COne.

\newcommand{\GammaExpectationKeyLemma}{
Consider any probability vector $p$ satisfying condition \COne for constant $\delta \in (0, 1)$ and (not necessarily constant) $\eps \in (0, 1)$, and any sorted load vector $x \in \R^n$ with $\Phi :=\Phi(\gamma)$, $\Psi :=\Psi(\gamma)$ and $\Gamma :=\Gamma(\gamma)$ for any smoothing parameter $\gamma \in (0, 1]$. Further define,
\[
\Delta\overline{\Phi} := \sum_{i=1}^n \Delta\overline{\Phi}_i := \sum_{i = 1}^n\Phi_i \cdot \left(p_i - \frac{1}{n}\right) \cdot \gamma,
\quad 
\text{ and }
\quad
\Delta\overline{\Psi} := \sum_{i=1}^n \Delta\overline{\Psi}_i := \sum_{i = 1}^n \Psi_i \cdot \left(\frac{1}{n} - p_i\right) \cdot \gamma.
\]
Then, there exists a constant $c := c(\delta) > 0$, such that
\[
\Delta\overline{\Gamma} := \sum_{i = 1}^n \Delta\overline{\Gamma}_i := \sum_{i = 1}^n \Delta\overline{\Phi}_i + \Delta\overline{\Psi}_i \leq -\Gamma \cdot \frac{\gamma\eps\delta}{4n} + c\gamma\eps.
\]}

\begin{lem}[Key Lemma] \label{thm:main_ptw}
\GammaExpectationKeyLemma
\end{lem}

\section{Proof of Main Theorem (Theorem~\ref{thm:hyperbolic_cosine_expectation})} \label{sec:drift_thm_proof}

In this section, we prove the main theorem (\cref{thm:hyperbolic_cosine_expectation}). We start with an outline for the proof of the key lemma (\cref{thm:main_ptw}), which we later prove. Finally, we apply the key lemma to derive the main theorem.

Before presenting the proof, we outline the key ideas in the proof: 

\begin{enumerate}
    \item It suffices to analyze $\Delta\overline{\Gamma} = \Delta\overline{\Phi} + \Delta\overline{\Psi}$ for the simplified probability vector,
\begin{align}  \label{eq:gamma_worst_case_vector}
r_i := \begin{cases}
 \frac{1 - \eps}{n} & \text{if } i \leq \delta n, \\
 \frac{1 + \tilde{\eps}}{n} & \text{otherwise},
\end{cases}
\end{align}
where $\tilde{\eps} := \eps \cdot \frac{\delta}{1 - \delta}$, as $r$ maximizes the terms $\Delta\overline{\Phi}$ and $\Delta\overline{\Psi}$, over all probability vectors satisfying condition \COne for a given $\delta$ and $\eps$.

\item For any bin $i \in [n]$, there is one dominant term in $\Gamma_i = \Phi_i + \Psi_i = e^{\gamma y_i} + e^{-\gamma y_i}$: for overloaded bins ($y_i \geq 0$) it is $\Phi_i = e^{\gamma y_i}$ (and $\Psi_i = e^{-\gamma y_i} \leq 1$) and for underloaded bins ($y_i < 0$) it is $\Psi_i$ (and $\Phi_i \leq 1$). In \cref{clm:change_of_dominant}, we show that the contribution of the non-dominant term in $\Delta\overline{\Gamma}$ is subsumed by the additive term, i.e., $c\gamma\eps$. 

\item Any overloaded bin $i \in [n]$ with $i \leq \delta n$, satisfies $r_i = \frac{1-\eps}{n}$ and so $\Delta\overline{\Phi}_i = -\Phi_i \cdot \frac{\gamma\eps}{n}$. We call these the set $\mathcal{G}_+$ of \textit{good overloaded bins}, as their dominant term decreases in expectation. The rest of the overloaded bins are the \textit{bad overloaded bins} $\mathcal{B}_+$, as these satisfy $\Delta\overline{\Phi}_i = + \Phi_i \cdot \frac{\gamma\tilde{\eps}}{n}$. 

Similarly, \textit{good underloaded bins} $\mathcal{G}_-$ with $i > \delta n$, satisfy $\Delta\overline{\Psi}_i = -\Psi_i \cdot \frac{\gamma\tilde{\eps}}{n}$ and \textit{bad underloaded bins} $\mathcal{B}_-$ satisfy $\Delta\overline{\Psi}_i = +\Psi_i \cdot \frac{\gamma\eps}{n}$.

\definecolor{DecTerm}{RGB}{84,130,53}
\definecolor{IncTerm}{RGB}{222,0,0}
\begin{table}[H]
\renewcommand{\arraystretch}{1.5}
    \centering
    \scalebox{0.8}{
    \begin{tabular}{|c|c|c|c|c|}
    \hline
      Set & Load & Index & $r_i$ & Contribution $\Delta\overline{\Gamma}_i$ \\ \hline
    $\mathcal{G}_+$ & $y_i \geq 0$ & $i \leq \delta n$ & $\frac{1-\eps}{n}$ & $\textcolor{DecTerm}{-\Phi_i \cdot \frac{\gamma\eps}{n}} + \Psi_i \cdot \frac{\gamma\eps}{n}$ \\
    $\mathcal{B}_+$ & $y_i \geq 0$ & $i > \delta n$ & $\frac{1+\tilde{\eps}}{n}$ & $\textcolor{IncTerm}{+\Phi_i \cdot \frac{\gamma\tilde{\eps}}{n}} - \Psi_i \cdot \frac{\gamma\tilde{\eps}}{n}$ \\
    $\mathcal{G}_-$ & $y_i < 0$ & $i > \delta n$ & $\frac{1+\tilde{\eps}}{n}$ & $+\Phi_i \cdot \frac{\gamma\tilde{\eps}}{n} \textcolor{DecTerm}{- \Psi_i \cdot \frac{\gamma\tilde{\eps}}{n}}$ \\
    $\mathcal{B}_-$ & $y_i < 0$ & $i \leq \delta n$ & $\frac{1-\eps}{n}$ & $-\Phi_i \cdot \frac{\gamma\eps}{n} \textcolor{IncTerm}{+ \Psi_i \cdot \frac{\gamma\eps}{n}}$ \\ \hline
    \end{tabular}}
    \caption{The definition of the four sets of bins and the contribution term of each bin to $\Delta\overline{\Gamma}$. The \textit{dominant term} is colored. The sign of the dominant term determines if a bin is good (\textcolor{DecTerm}{negative sign/decrease}) or bad (\textcolor{IncTerm}{positive sign/increase}).}
    \label{tab:four_sets_of_bins}
\end{table}

\item We can either have $\mathcal{B}_+ \neq \emptyset$ or $\mathcal{B}_- \neq \emptyset$ (see \cref{fig:general_case}). 

The handling of one case is symmetric to the other due to the symmetric nature of $\Delta \overline{\Phi}$ and $\Delta \overline{\Psi}$ (with $\delta$ being replaced by $1-\delta$). So, from here on we only consider the case with $\mathcal{B}_+ \neq \emptyset$ (and $\mathcal{B}_- = \emptyset$).

\item \textbf{Case A.1:} When the number of bad overloaded bins is small (i.e., $1 \leq |\mathcal{B}_+| \leq \frac{n}{2} \cdot (1 - \delta)$), the positive contribution of the bins in $\mathcal{B}_+$ is counteracted by the negative contribution of the bins in $\mathcal{G}_+$ (\cref{fig:case_a}). We prove this by making the worst-case assumption that all bad bins have load  $y_{\delta n}$. All underloaded bins are good, so on aggregate we get a decrease.

\item \textbf{Case A.2:} Consider the case when the number of bad overloaded bins is large $|\mathcal{B}_+| > \frac{n}{2} \cdot (1 - \delta)$. The positive contribution of the first $\frac{n}{2} \cdot (1 - \delta)$ of the bins $\mathcal{B}_+$, call them $\mathcal{B}_1$, is counteracted by the negative contribution of the bins in $\mathcal{G}_+$ as in Case A.1. The positive contribution of the remaining bad bins $\mathcal{B}_2$ is counteracted by a fraction of the negative contribution of the bins in $\mathcal{G}_-$. This is because the number of ``holes'' (empty ball slots in the underloaded bins) in the bins of $\mathcal{G}_-$ are significantly more than the number of balls in $\mathcal{B}_2$. Hence, again on aggregate we get a decrease (\cref{fig:case_b_2}).

\end{enumerate}

\begin{figure}
    \centering
    \includegraphics[scale=0.55]{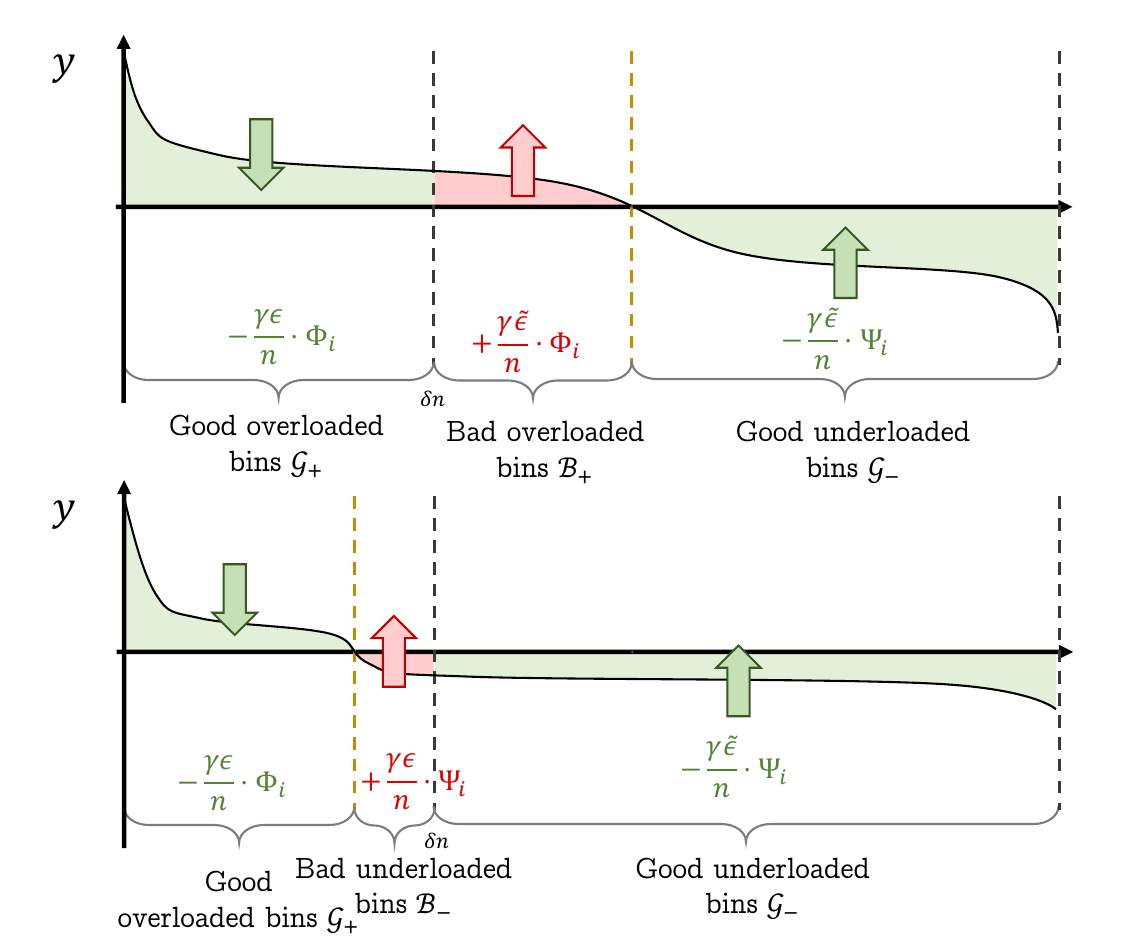}
    \caption{The two cases of bad bins ($\mathcal{B}_+ \neq \emptyset$ and $\mathcal{B}_- \neq \emptyset$ respectively) in a load vector and their dominant terms in $\Delta\overline{\Gamma}$ for each of the set of bins. The dominant terms that are decreasing are shown in green and the dominant terms that are increasing are shown in red.}
    \label{fig:general_case}
\end{figure}

We proceed with a simple claim for bounding the contributions of the non-dominant terms:

\begin{clm} \label{clm:change_of_dominant}
Consider the probability vector $r$ as defined in \cref{eq:gamma_worst_case_vector}. For any bin $i \in [n]$ with $y_i \geq 0$, we have that\[
\Delta\overline{\Psi}_i \leq - \Psi_i \cdot \frac{\gamma\eps\delta}{4n} + \frac{2\gamma\eps}{n},
\]
and for any bin $i \in [n]$ with $y_i < 0$, we have that\[
\Delta\overline{\Phi}_i \leq -\Phi_i \cdot \frac{\gamma\eps\delta}{4n} + \frac{2\gamma\tilde{\eps}}{n}.
\]
\end{clm}
\begin{proof}
For any bin $i \in [n]$ with $y_i \geq 0$, we have that
\[
\Delta\overline{\Psi}_i 
 \leq \max\left\{+\Psi_i \cdot \frac{\gamma\eps}{n}, -\Psi_i \cdot \frac{\gamma\tilde{\eps}}{n} \right\} 
 = - \Psi_i \cdot \frac{\gamma\eps\delta}{4n} + \Psi_i \cdot \frac{\gamma}{n} \cdot \left( \frac{\eps\delta}{4} + \eps\right) 
 \leq - \Psi_i \cdot \frac{\gamma\eps\delta}{4n} + \Psi_i \cdot \frac{2\gamma\eps}{n},
\]
using that $\delta \leq 1$.

Similarly, for any bin $i \in [n]$ with $y_i < 0$, we have that
\[
\Delta\overline{\Phi}_i 
 \leq \max\left\{ +\Phi_i \cdot \frac{\gamma\tilde{\eps}}{n}, -\Phi_i \cdot \frac{\gamma\eps}{n} \right\} 
 = - \Phi_i \cdot \frac{\gamma\eps\delta}{4n} + \Phi_i \cdot \frac{\gamma}{n} \cdot \left( \frac{\eps\delta}{4} + \tilde{\eps}\right) \leq - \Phi_i \cdot \frac{\gamma\eps\delta}{4n} + \Phi_i \cdot \frac{2\gamma\tilde{\eps}}{n},
\]
using that $\eps\delta \leq \frac{\eps\delta}{1 - \delta} = \tilde{\eps}$.
\end{proof}

We now turn to the proof of \cref{thm:main_ptw}.

\begin{proof}[Proof of \cref{thm:main_ptw}]
Fix a labeling of the bins so that they are sorted non-increasingly according to their load in $x$. Let $p$ be the probability vector satisfying condition \COne for some $\epsilon \in (0,1)$ and $\delta \in \{ 1/n, \ldots, 1 \}$. Recall that the probability vector $r$ was defined as,
\begin{align*}
r_i := \begin{cases}
 \frac{1 - \eps}{n} & \text{if } i \leq \delta n, \\
 \frac{1 + \tilde{\eps}}{n} & \text{otherwise},
\end{cases} 
\end{align*}
where $\tilde{\eps} := \eps \cdot \frac{\delta}{1 - \delta}$. Thanks to the definition of $\tilde{\epsilon}$, it is clear that $r$ is also a probability vector. Further, for any $1 \leq k \leq \delta n$, due to condition \COne,
\[
 \sum_{i=1}^k p_i \leq (1 - \eps) \cdot \frac{k}{n} = \sum_{i=1}^k r_i, 
\]
and any $\delta n + 1 \leq k \leq n$,
\[
 \sum_{i=k}^{n} p_i \geq \left( 1 + \eps \cdot \frac{\delta}{1 - \delta} \right) \cdot \frac{n - k + 1}{n} = \sum_{i=k}^{n} r_i.
\]
This implies that $p$ is majorized by $r$.
Since $\Phi_i$ (and $\Psi_i$) are non-increasing (and non-decreasing) in $i \in [n]$, using \cref{lem:quasilem2}, the terms 
\[
\Delta\overline{\Phi} = \sum_{i = 1}^n \Phi_i \cdot \left(p_i - \frac{1}{n}\right) \cdot \gamma,
\quad 
\text{and}
\quad
\Delta\overline{\Psi} = \sum_{i = 1}^n \Psi_i \cdot \left(\frac{1}{n} - p_i\right) \cdot \gamma,
\]
are at least as large for $r$ than for $p$.
Hence, from now on, we will be working with $p = r$.

Recall that we partition overloaded bins $i$ with $y_i \geq 0$ into \textit{good overloaded bins} $\mathcal{G}_+$ with $p_i = \frac{1 - \eps}{n}$ and into \textit{bad overloaded bins} $\mathcal{B}_+$ with $p_i = \frac{1 + \tilde{\eps}}{n}$ (see \cref{tab:four_sets_of_bins}). These are called good bins, because any bin $i \in \mathcal{G}_+$ satisfies $\Delta\overline{\Phi}_i = -\Phi_i \cdot \frac{\gamma\eps}{n}$ and since $\Psi_i \leq 1$ for overloaded bins, this implies overall a drop in expectation for $\Gamma_i$. %

\textbf{Case A [$\mathcal{B}_+ \neq \emptyset$]:} We partition $\mathcal{B}_+$ into $\mathcal{B}_1 := \mathcal{B}_+ \cap \{i \in [n]: i \leq \frac{n}{2} \cdot (1+\delta) \}$ and $\mathcal{B}_2 := \mathcal{B}_+ \setminus \mathcal{B}_1$. In Case A.1 we handle the case where $\mathcal{B}_2 = \emptyset$ and in Case A.2 the case where $\mathcal{B}_2 \neq \emptyset$. Finally, in Case B we handle the case where $\mathcal{B}_- \neq \emptyset$ by a symmetry argument.

\begin{figure}[H]
    \centering
    \includegraphics[scale=0.55]{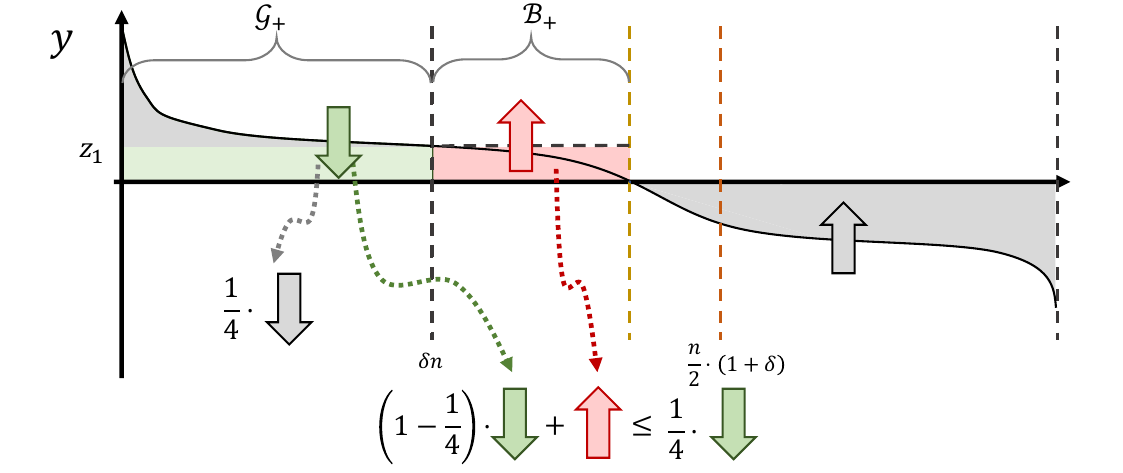}
    \caption{The dominant contributions for the bins in $\mathcal{G_+}$, $\mathcal{B}_+$ and $\mathcal{G}_-$ in Case A. In the analysis of Case A.1, the positive contribution of the bins in $\mathcal{B}_+$ (shown in red) is counteracted by a fraction of the negative contribution of the good bins $\mathcal{G}_+$ shown in green, (pessimistically) assuming they all have load $z_1$. In gray are the remaining negative contributions of the good bins. }
    \label{fig:case_a}
\end{figure}

\textbf{Case A.1 [$1 \leq |\mathcal{B}_+| \leq \frac{n}{2} \cdot (1 - \delta)$]:} Intuitively, in this case there are not many bad bins in $\mathcal{B}_+$, so their (positive) contribution is counteracted by the (negative) contribution of good bins in $\mathcal{G}_+$ (see \cref{fig:case_a}). To formalize this, let $z_1 := y_{\delta n}$ (by assumption, $|\mathcal{B}_{+}| \geq 1$ and so $z_1 \geq 0$). Then, $y_i \geq z_1$ for any bin $i \in \mathcal{G}_+$, and $y_i \leq z_1$, for any $i \in \mathcal{B}_+$. With some foresight, we use $\mathcal{B}_1$ instead of $\mathcal{B}_+$, since in this case $\mathcal{B}_1 = \mathcal{B}_+$ and it will also allow us to use \cref{eq:case_a1_good_and_bad_combined} in Case A.2. Hence,
\begin{align}
 \sum_{i \in \mathcal{B}_1} \Delta\overline{\Phi}_i
 & = \sum_{i \in \mathcal{B}_1} \Phi_i \cdot  \frac{\gamma\tilde{\eps}}{n} 
   = -\sum_{i \in \mathcal{B}_1} \Phi_i \cdot  \frac{\gamma\eps\delta}{4n} + \sum_{i \in \mathcal{B}_1} \Phi_i \cdot  \frac{\gamma}{n} \cdot \left( \tilde{\eps} + \frac{\eps\delta}{4} \right) \notag \\
 & \stackrel{(a)}{\leq} -\sum_{i \in \mathcal{B}_1} \Phi_i \cdot  \frac{\gamma\eps\delta}{4n} + \sum_{i \in \mathcal{B}_1} \Phi_i \cdot  \frac{3\gamma\tilde{\eps}}{2n} \notag \\
 & \stackrel{(b)}{\leq} -\sum_{i \in \mathcal{B}_1} \Phi_i \cdot  \frac{\gamma\eps\delta}{4n} + \frac{n}{2} \cdot (1-\delta) \cdot e^{\gamma z_1} \cdot  \frac{3\gamma\tilde{\eps}}{2n} \notag \\
 & \stackrel{(c)}{=} -\sum_{i \in \mathcal{B}_1} \Phi_i \cdot  \frac{\gamma\eps\delta}{4n} + e^{\gamma z_1} \cdot  \frac{3\gamma\eps\delta}{4}, \label{eq:b_plus_change}
\end{align}
using in $(a)$ that $\eps \delta \leq \tilde{\eps}$ and in $(b)$ that $y_i \leq z_1$ for $i \in \mathcal{B}_1$ (and so $\Phi_i \leq e^{\gamma z_1}$) and $|\mathcal{B}_1| \leq \frac{n}{2} \cdot (1 - \delta)$, and in $(c)$ that $\tilde{\eps} = \frac{\eps\delta}{1 - \delta}$. For bins in $\mathcal{G}_+$,
\begin{align}
\sum_{i \in \mathcal{G}_+} \Delta\overline{\Phi}_i  & = - \sum_{i \in \mathcal{G}_+} \Phi_i \cdot \frac{\gamma\eps}{n}
 = -\sum_{i \in \mathcal{G}_+} \Phi_i \cdot \frac{\gamma\eps}{4n} -\sum_{i \in \mathcal{G}_+} \Phi_i \cdot \frac{3\gamma\eps}{4n} \notag \\
 & \stackrel{(a)}{\leq} -\sum_{i \in \mathcal{G}_+} \Phi_i \cdot \frac{\gamma\eps}{4n} -\sum_{i \in \mathcal{G}_+} e^{\gamma z_1} \cdot \frac{3\gamma\eps}{4n} 
 \stackrel{(b)}{=} -\sum_{i \in \mathcal{G}_+} \Phi_i \cdot \frac{\gamma\eps}{4n} - e^{\gamma z_1} \cdot \frac{3\gamma\eps\delta}{4}, \label{eq:g_plus_change}
\end{align}
using in $(a)$ that $y_i \geq z_1$ for any $i \in \mathcal{G}_+$ and in $(b)$ that $|\mathcal{G}_+| = \delta n$, since $|\mathcal{B}_1| \geq 1$.

Hence, combining \cref{eq:b_plus_change} and \cref{eq:g_plus_change}, the contribution of overloaded bins to $\Phi$ is given by 
\begin{align}
\sum_{i \in \mathcal{G}_+ \cup \mathcal{B}_1} \Delta\overline{\Phi}_i
 & \leq -\sum_{i \in \mathcal{G}_+} \Phi_i \cdot \frac{\gamma\eps}{4n} -\sum_{i \in \mathcal{B}_1} \Phi_i \cdot  \frac{\gamma\eps\delta}{4n}
 \leq -\sum_{i \in \mathcal{G}_+ \cup \mathcal{B}_1} \Phi_i \cdot  \frac{\gamma\eps\delta}{4n}\label{eq:case_a1_good_and_bad_combined}.
\end{align}
Therefore, aggregating the contributions to $\Delta\overline{\Gamma}$ as described above, we get that
\begin{align*}
\Delta\overline{\Gamma} 
 & = \sum_{i \in \mathcal{G}_+} \Delta\overline{\Gamma}_i + \sum_{i \in \mathcal{B}_1} \Delta\overline{\Gamma}_i + \sum_{i \in \mathcal{G}_-} \Delta\overline{\Gamma}_i \\
 & = \sum_{i \in \mathcal{G}_+ \cup \mathcal{B}_1} \Delta\overline{\Phi}_i + \sum_{i \in \mathcal{G}_-} \Delta\overline{\Psi}_i + \sum_{i \in \mathcal{G}_+ \cup \mathcal{B}_1} \Delta\overline{\Psi}_i + \sum_{i \in \mathcal{G}_-} \Delta\overline{\Phi}_i \\
 & \!\!\!\! \stackrel{(\ref{eq:case_a1_good_and_bad_combined})}{\leq} - \sum_{i \in \mathcal{G}_+ \cup \mathcal{B}_1} \Phi_i \cdot \frac{\gamma\eps\delta}{4n} - \sum_{i \in \mathcal{G}_-} \Psi_i \cdot \frac{\gamma\tilde{\eps}}{n} + \sum_{i \in \mathcal{G}_+ \cup \mathcal{B}_1} \Delta\overline{\Psi}_i + \sum_{i \in \mathcal{G}_-} \Delta\overline{\Phi}_i \\
 & \stackrel{(a)}{\leq} - \sum_{i \in \mathcal{G}_+ \cup \mathcal{B}_1} \Phi_i \cdot \frac{\gamma\eps\delta}{4n} - \sum_{i \in \mathcal{G}_-} \Psi_i \cdot \frac{\gamma\eps\delta}{4n} - \sum_{i \in \mathcal{G}_+ \cup \mathcal{B}_1} \Psi_i \cdot \frac{\gamma\eps\delta}{4n} - \sum_{i \in \mathcal{G}_-} \Phi_i \cdot \frac{\gamma\eps\delta}{4n} + \sum_{i = 1}^n \frac{2\gamma}{n} \cdot \max\{\eps, \tilde{\eps} \} \\
 & \stackrel{(b)}{=} - \sum_{i = 1}^n \Gamma_i \cdot \frac{\gamma\eps\delta}{4n} + 2\gamma \cdot \max\{\eps, \tilde{\eps} \}
 = - \Gamma \cdot \frac{\gamma\eps\delta}{4n} + 2\gamma \cdot \max\{\eps, \tilde{\eps} \},
\end{align*}
using in $(a)$ that $\eps\delta \leq \tilde{\eps}$ and \cref{clm:change_of_dominant} to bound the contributions of the non-dominant terms and in $(b)$ that $\Gamma_i = \Phi_i + \Psi_i$ for any bin $i \in [n]$.

\begin{figure}[H]
    \centering
    \includegraphics[scale=0.55]{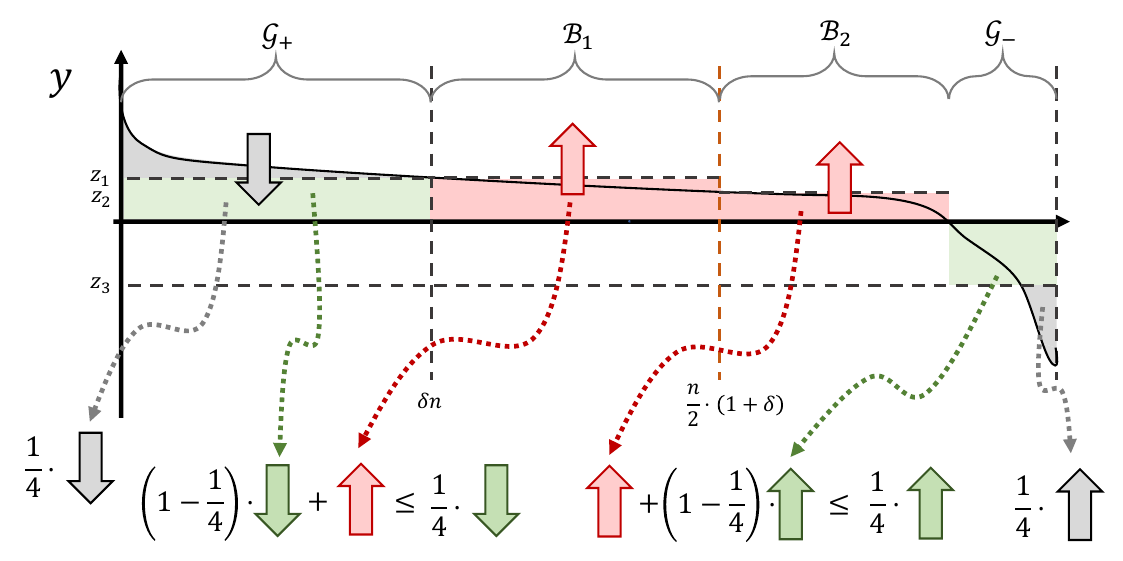}
    \caption{Case A.2: The dominant (positive) contribution of the bins in $\mathcal{B}_1$ (shown in red) is counteracted by a fraction of the dominant (negative) contribution of the bins in $\mathcal{G}_+$ (shown in green) as in Case A.1. The dominant (negative) contribution of the bins in $\mathcal{B}_2$ is counteracted by a fraction of the dominant (negative) contribution of the bins in $\mathcal{G}_-$ (shown in green), when $z_2$ is sufficiently large. In grey, the dominant decrease terms for $\mathcal{G}_+$ and $\mathcal{G}_-$ (which do not contribute to counteracting the increase).}
    \label{fig:case_b_2}
\end{figure}

\textbf{Case A.2 [$|\mathcal{B}_+| > \frac{n}{2} \cdot (1 - \delta)$]:} Recall that we partitioned the bins $\mathcal{B}_+$ into $\mathcal{B}_1 := \mathcal{B}_+ \cap \{i \in [n]: i \leq \frac{n}{2} \cdot (1+\delta) \}$ and $\mathcal{B}_2 := \mathcal{B}_+ \setminus \mathcal{B}_1$. We will counteract the positive contribution $\Delta\overline{\Phi}_i$ for bins $i \in \mathcal{B}_1$ by the negative contribution of the bins in $\mathcal{G}_+$ as in \cref{eq:case_a1_good_and_bad_combined} in Case A.1. For that of bins in $\mathcal{B}_2$ we will consider two cases based on $z_2 := y_{\frac{n}{2} \cdot (1 + \delta)} > 0$, the load of the heaviest bin in $\mathcal{B}_2$. Similarly to \cref{clm:change_of_dominant}, we obtain a bound for the dominant contribution of the bins in $\mathcal{B}_2$
\begin{align} 
\sum_{i \in \mathcal{B}_2} \Delta\overline{\Phi}_i
 & = \sum_{i \in \mathcal{B}_2} \Phi_i \cdot \frac{\gamma\tilde{\eps}}{n} 
 = - \sum_{i \in \mathcal{B}_2} \Phi_i \cdot \frac{\gamma\eps\delta}{4n} + \sum_{i \in \mathcal{B}_2} \Phi_i \cdot \frac{\gamma}{n} \cdot \left( \frac{\eps\delta}{4} + \tilde{\eps} \right) \notag \\
 & \leq - \sum_{i \in \mathcal{B}_2} \Phi_i \cdot \frac{\gamma\eps\delta}{4n} + \sum_{i \in \mathcal{B}_2} \Phi_i \cdot \frac{2\gamma\tilde{\eps}}{n} \notag \\
 & \stackrel{(a)}{\leq} - \sum_{i \in \mathcal{B}_2} \Phi_i \cdot \frac{\gamma\eps\delta}{4n} + |\mathcal{B}_2| \cdot e^{\gamma z_2} \cdot \frac{2\gamma\tilde{\eps}}{n} \notag \\
 & \stackrel{(b)}{\leq} - \sum_{i \in \mathcal{B}_2} \Phi_i \cdot \frac{\gamma\eps\delta}{4n} + e^{\gamma z_2} \cdot (\gamma\eps\delta), \label{eq:b2_contribution}
\end{align}
using in $(a)$ that $y_i \leq z_2$ for $i \in \mathcal{B}_2$ and in $(b)$ that $|\mathcal{B}_2| \leq \frac{n}{2} \cdot (1 - \delta)$ and $\tilde{\eps} = \frac{\eps\delta}{1 - \delta}$.

\textbf{Case A.2.1 [$z_2 \leq \frac{1}{\gamma} \cdot \frac{1 - \delta}{2\delta} \cdot \log(8/3)$]:} In this case, the loads of the bins in $\mathcal{B}_2$ are small enough for their contribution to be counteracted by the additive term. More precisely, we get that %
\begin{align} \label{eq:case_b1_B2_bound}
\sum_{i \in \mathcal{B}_2} \Delta\overline{\Phi}_i
 & \stackrel{(\ref{eq:b2_contribution})}{\leq} - \sum_{i \in \mathcal{B}_2} \Phi_i \cdot \frac{\gamma\eps\delta}{4n} + e^{\gamma z_2} \cdot (\gamma\eps\delta) \leq - \sum_{i \in \mathcal{B}_2} \Phi_i \cdot \frac{\gamma\eps\delta}{4n} + e^{\frac{1 - \delta}{2\delta} \cdot \log(8/3)} \cdot (\gamma\eps\delta).
\end{align}

Hence, we can now aggregate the contributions as follows
\begin{align*}
\Delta\overline{\Gamma} 
 & = \sum_{i \in \mathcal{G}_+} \Delta\overline{\Gamma}_i + \sum_{i \in \mathcal{B}_1} \Delta\overline{\Gamma}_i + \sum_{i \in \mathcal{B}_2} \Delta\overline{\Gamma}_i + \sum_{i \in \mathcal{G}_-} \Delta\overline{\Gamma}_i \\
 & = \sum_{i \in \mathcal{G}_+ \cup \mathcal{B}_1} \Delta\overline{\Phi}_i + \sum_{i \in \mathcal{B}_2} \Delta\overline{\Phi}_i + \sum_{i \in \mathcal{G}_-} \Delta\overline{\Psi}_i + \sum_{i \in \mathcal{G}_+ \cup \mathcal{B}_+} \Delta\overline{\Psi}_i + \sum_{i \in \mathcal{G}_-} \Delta\overline{\Phi}_i \\
 & \stackrel{(a)}{\leq} - \sum_{i \in \mathcal{G}_+ \cup \mathcal{B}_1} \Phi_i \cdot \frac{\gamma\eps\delta}{4n} 
 - \sum_{i \in \mathcal{B}_2} \Phi_i \cdot \frac{\gamma\eps\delta}{4n}
 + e^{ \frac{1 - \delta}{2\delta} \cdot \log(8/3)} \cdot (\gamma\eps\delta)
 - \sum_{i \in \mathcal{G}_-} \Psi_i \cdot \frac{\gamma\eps\delta}{4n} \\
 & \qquad - \sum_{i \in \mathcal{G}_+ \cup \mathcal{B}_+} \Psi_i \cdot \frac{\gamma\eps\delta}{4n} - \sum_{i \in \mathcal{G}_-} \Phi_i \cdot \frac{\gamma\eps\delta}{4n} + \sum_{i = 1}^n \frac{2\gamma}{n} \cdot \max\{\eps, \tilde{\eps} \} \\
 & \leq - \sum_{i = 1}^n \Gamma_i \cdot \frac{\gamma\eps\delta}{4n} +  4\gamma \cdot \max\left\{\eps, \tilde{\eps}, \eps\delta \cdot e^{ \frac{1 - \delta}{2\delta} \cdot \log(8/3)} \right\},
\end{align*}
using in $(a)$: $(i)$ \cref{eq:case_a1_good_and_bad_combined} for bounding the contribution of bins in $\mathcal{B}_1$, $(ii)$ the \cref{eq:case_b1_B2_bound} for bounding the contribution of bins in $\mathcal{B}_2$ and $(iii)$ the \cref{clm:change_of_dominant} for bounding the contributions of the non-dominant terms.

\textbf{Case A.2.2 [$z_2 > \frac{1}{\gamma} \cdot \frac{1 - \delta}{2\delta} \cdot \log(8/3)$]:} In this case, $z_2$ being large means that there are substantially more holes (ball slots below the average line) in the underloaded bins than balls in the overloaded bins of $\mathcal{B}_1$. Hence, as we will prove below the negative contribution $\Delta\overline{\Psi}$ for bins in $\mathcal{G}_-$ counteracts the positive contribution of $\Delta\overline{\Phi}$ for $\mathcal{B}_1$ (\cref{fig:case_b_2}). 

Next note that because $\Psi_i$ is non-decreasing in $i \in [n]$, the term $\sum_{i \in \mathcal{G}_-} \Psi_i$ is minimised when all underloaded bins are equal to the same load $-z_3 < 0$, i.e., $\sum_{i \in \mathcal{G}_-} \Psi_i \geq |\mathcal{G}_-| \cdot e^{\gamma z_3}$.
Further, note that $z_3 \geq \frac{z_2 \cdot (|\mathcal{G}_+| + |\mathcal{B}_+|)}{|\mathcal{G}_-|} \geq \frac{z_2 \cdot \frac{n}{2} \cdot (1+\delta)}{|\mathcal{G}_-|}$ by the assumption $|\mathcal{B}_+| > \frac{n}{2} \cdot (1-\delta)$ and $|\mathcal{G}_+| = \delta n$, and therefore,
\[
 \sum_{i \in \mathcal{G}_-} \Psi_i \geq |\mathcal{G}_-| \cdot e^{\gamma \cdot \frac{z_2 \cdot \frac{n}{2} \cdot (1+\delta)}{|\mathcal{G}_-|}} =: g(|\mathcal{G}_-|),
\]
where we seek to lower bound the function $g:[1,n] \rightarrow \mathbb{R}$. To this end, we will first upper bound $|\mathcal{G}_-|$, using the assumption for Case A.2, 
\[
|\mathcal{G}_-| = n -  |\mathcal{G}_+| - |\mathcal{B}_+| \leq n - n\delta - \frac{n}{2} \cdot (1 - \delta) = \frac{n}{2} \cdot (1 - \delta).
\]
Further, $\frac{n}{2} \cdot (1-\delta) \leq \gamma \cdot z_2 \cdot \frac{n}{2} \cdot (1+\delta) =: M$, by the precondition on $z_2$ (and that $2\delta \leq 1 + \delta$), and so we also have $|\mathcal{G}_-| \leq M$.
By \cref{lem:decreasing_fn}, the function $f(x) = x \cdot e^{k/x}$ is decreasing for $0 < x \leq k$, and so $g$ is decreasing for $1 < |\mathcal{G}_-| \leq \frac{n}{2} \cdot (1-\delta) \leq M$.
Hence, $g(|\mathcal{G}_-|)$ is minimised by $|\mathcal{G}_-| = \frac{n}{2} \cdot (1-\delta)$. Therefore,
\begin{align}
 \sum_{i \in \mathcal{G}_-} \Delta\overline{\Psi}_i 
 = - \sum_{i \in \mathcal{G}_-} \Psi_i \cdot \frac{\gamma \tilde{\epsilon}}{n} %
 \geq -\min_{|\mathcal{G}_-| \in [1,\frac{n}{2} \cdot (1-\delta)]} |\mathcal{G}_-| \cdot e^{\frac{M}{|\mathcal{G}_-|}} \cdot \frac{\gamma \tilde{\epsilon}}{n}
 = -\frac{n}{2} \cdot (1-\delta) \cdot e^{\frac{M}{\frac{n}{2} \cdot (1-\delta)}} \cdot \frac{\gamma \tilde{\epsilon}}{n}. %
 \label{eq:previous}
\end{align}
We can lower bound the exponent of the last term as follows,
\[
\frac{M}{\frac{n}{2} \cdot (1 - \delta)} = \frac{\gamma z_2 \cdot (1 + \delta)}{1 - \delta} = \gamma z_2 + \gamma z_2 \cdot \frac{2\delta}{1 - \delta} \geq \gamma z_2 + \log(8/3),
\]
using the assumption that $z_2 > \frac{1}{\gamma} \cdot \frac{1 - \delta}{2\delta} \cdot \log(8/3)$. 

Now we will split the contributions of the bins in $\mathcal{G}_-$,
\begin{align}
\sum_{i \in \mathcal{G}_-} \Delta\overline{\Psi}_i  
 & = - \sum_{i \in \mathcal{G}_-} \Psi_i \cdot \frac{\gamma \tilde{\eps}}{n} 
 = -\sum_{i \in \mathcal{G}_-} \Psi_i \cdot \frac{\gamma \tilde{\eps}}{4n} - \sum_{i \in \mathcal{G}_-} \Psi_i \cdot \frac{3\gamma\tilde{\eps}}{4n} \notag \\
 & \!\!\!\! \stackrel{(\ref{eq:previous})}{\leq} -\sum_{i \in \mathcal{G}_-} \Psi_i \cdot \frac{\gamma \tilde{\eps}}{4n} - \frac{n}{2} \cdot (1-\delta) \cdot e^{\gamma z_2 + \log(8/3) } \cdot \frac{3\gamma\tilde{\eps}}{4n} \notag \\ 
 & = 
 -\sum_{i \in \mathcal{G}_-} \Psi_i \cdot \frac{\gamma \tilde{\eps}}{4n} - e^{\gamma z_2} \cdot (\gamma \eps \delta), \label{eq:g_minus_lb}
\end{align}
using in the last equality that $\tilde{\eps} = \frac{\eps\delta}{1 - \delta}$.

We will now show that the dominant increase for bins in $\mathcal{B}_2$ is counteracted by a fraction of the dominant decrease of those in $\mathcal{G}_-$. Combining \cref{eq:b2_contribution} and \cref{eq:g_minus_lb}
\begin{align}
\sum_{i \in \mathcal{B}_2} \Delta\overline{\Phi}_i + \sum_{i \in \mathcal{G}_-} \Delta\overline{\Psi}_i
 & \leq -\sum_{i \in \mathcal{B}_2} \Phi_i \cdot \frac{\gamma\eps\delta}{4n} + e^{\gamma z_2} \cdot (\gamma\eps\delta) -\sum_{i \in \mathcal{G}_-} \Psi_i \cdot \frac{\gamma \tilde{\eps}}{4n} - \sum_{i \in \mathcal{G}_-} \Psi_i \cdot \frac{3\gamma\tilde{\eps}}{4n} \notag \\
 & \!\!\!\! \stackrel{(\ref{eq:g_minus_lb})}{\leq}  -\sum_{i \in \mathcal{B}_2} \Phi_i \cdot \frac{\gamma\eps\delta}{4n} -\sum_{i \in \mathcal{G}_-} \Psi_i \cdot \frac{\gamma \tilde{\eps}}{4n} \leq -\sum_{i \in \mathcal{B}_2} \Phi_i \cdot \frac{\gamma\eps\delta}{4n} -\sum_{i \in \mathcal{G}_-} \Psi_i \cdot \frac{\gamma\eps\delta}{4n}. \label{eq:b2_and_g_minus}
\end{align}
Finally, overall the contributions are given by
\begin{align}
\Delta\overline{\Gamma}
 & = \sum_{i \in \mathcal{G}_+} \Delta\overline{\Gamma}_i + \sum_{i \in \mathcal{B}_1} \Delta\overline{\Gamma}_i + \sum_{i \in \mathcal{B}_2} \Delta\overline{\Gamma}_i + \sum_{i \in \mathcal{G}_-} \Delta\overline{\Gamma}_i \notag \\
 & = \sum_{i \in \mathcal{G}_+ \cup \mathcal{B}_1} \Delta\overline{\Phi}_i + \left(\sum_{i \in \mathcal{B}_2} \Delta\overline{\Phi}_i + \sum_{i \in \mathcal{G}_-} \Delta\overline{\Psi}_i\right) + \sum_{i \in \mathcal{G}_+ \cup \mathcal{B}_+} \Delta\overline{\Psi}_i + \sum_{i \in \mathcal{G}_-} \Delta\overline{\Phi}_i \notag \\
 & \stackrel{(a)}{\leq} - \sum_{i \in \mathcal{G}_+ \cup \mathcal{B}_+} \Phi_i \cdot \frac{\gamma\eps \delta}{4n} -\sum_{i \in \mathcal{B}_2} \Phi_i \cdot \frac{\gamma\eps\delta}{4n} -\sum_{i \in \mathcal{G}_-} \Psi_i \cdot \frac{\gamma\eps\delta}{4n} - \sum_{i \in \mathcal{G}_+ \cup \mathcal{B}_+} \Psi_i \cdot \frac{\gamma\eps\delta}{4n} - \sum_{i \in \mathcal{G}_-} \Phi_i \cdot \frac{\gamma\eps\delta}{4n} \notag \\
 & \qquad + \sum_{i = 1}^n \frac{2\gamma}{n} \cdot \max\{ \eps, \tilde{\eps} \} \notag \\
 & = - \Gamma \cdot \frac{\gamma\eps \delta}{4n} + 2\gamma \cdot \max\{ \eps, \tilde{\eps} \}, \notag
\end{align}
using in $(a)$ that $(i)$ \cref{eq:case_a1_good_and_bad_combined} for bounding the contribution of the bins in $\mathcal{B}_1$, $(ii)$ the \cref{eq:b2_and_g_minus} for bounding the contribution of the bins in $\mathcal{B}_2 \cup \mathcal{G}_-$ and $(iii)$ \cref{clm:change_of_dominant} for bounding the non-dominant terms.

Combining the three subcases for Case A we have that
\[
\Delta\overline{\Gamma} 
 \leq -\Gamma \cdot \frac{\gamma\eps \delta}{4n} + c',
\]
where 
\[
c' := \max\left\{4\eps, 4\tilde{\eps}, 4 \eps \delta \cdot e^{\frac{1-\delta}{2\delta} \cdot \log(8/3)} \right\} \cdot \gamma
 = 4 \cdot \max\Big\{1, \frac{\delta}{1-\delta},  \delta \cdot e^{\frac{1-\delta}{2\delta} \cdot \log(8/3)} \Big\} \cdot \gamma \eps.
\]

\begin{figure}
    \centering
    \includegraphics[scale=0.55]{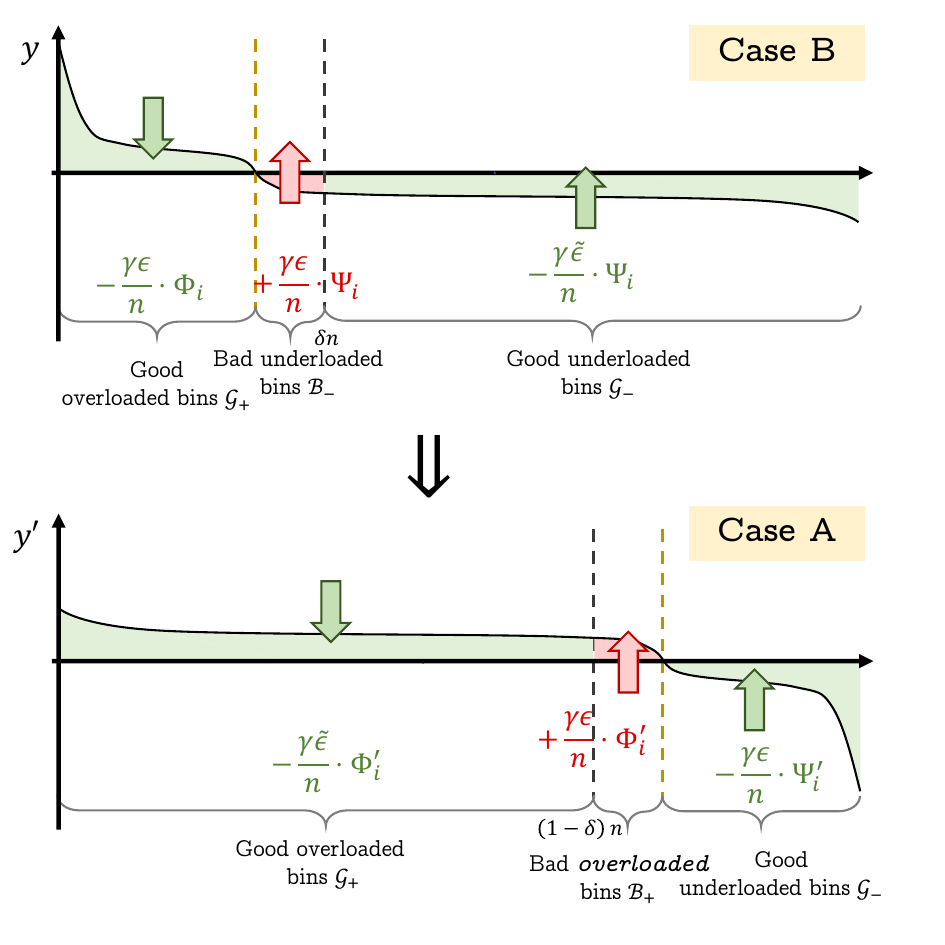}
    \caption{Visualization of how negating and sorting the vector $y$ to obtain $y'$, leads from Case $B$ to Case A with $\delta \to 1-\delta$ and $\eps \to \tilde{\eps}$. Also note that $\Phi' = \Psi$, $\Psi' = \Phi$ and $\Gamma' = \Gamma$.}
    \label{fig:symmetric_case}
\end{figure}

\textbf{Case B [$\mathcal{B}_- \neq \emptyset$]:} This case is symmetric to Case A, by interchanging $\Phi$ with $\Psi$, $\delta$ with $1-\delta$, $\eps \delta$ with $\tilde{\eps} (1 - \delta)$, and negating and sorting the normalized load vector (in non-increasing order). In particular, the three sub-cases are:
\begin{itemize}\itemsep0pt
  \item \textbf{Case B.1 [$1 \leq |\mathcal{B}_-| \leq \frac{n}{2} \cdot \delta$]}
  \item \textbf{Case B.2.1 [$|\mathcal{B}_-| > \frac{n}{2} \cdot \delta$, $z_2' \leq \frac{1}{\gamma} \cdot \frac{\delta}{2 \cdot (1 - \delta)} \cdot \log(8/3)$]} where $z_2' := y_{\frac{n}{2} \cdot \delta}$
  \item \textbf{Case B.2.2 [$|\mathcal{B}_-| > \frac{n}{2} \cdot \delta$, $z_2' > \frac{1}{\gamma} \cdot \frac{\delta}{2 \cdot (1 - \delta)} \cdot \log(8/3)$]}
\end{itemize}
Combining the three subcases for Case B we have that
\[
\Delta\overline{\Gamma} 
 \leq -\Gamma \cdot \frac{\gamma\eps \delta}{4n} + c'',
\]
where 
\[
c'' := \max\left\{4\eps, 4\tilde{\eps}, 4 \tilde{\eps} (1-\delta) \cdot e^{\frac{\delta}{2(1 - \delta)} \cdot \log(8/3)} \right\} \cdot \gamma
 = 4 \cdot \max\Big\{1, \frac{\delta}{1-\delta}, \delta \cdot e^{\frac{\delta}{2(1 - \delta)} \cdot \log(8/3)} \Big\} \cdot \gamma \eps,
\]
recalling that $\eps \delta = \tilde{\eps} (1-\delta)$.

Combining the Case A and Case B, we get that
\[
\Delta\overline{\Gamma} 
 \leq -\Gamma \cdot \frac{\gamma\eps \delta}{4n} + c\gamma\eps,
\]
where $c := 4 \cdot \max\Big\{ 1, \frac{\delta}{1 - \delta}, e^{ \frac{1 - \delta}{2\delta} \cdot \log(8/3)} \cdot \frac{\delta}{1-\delta}, \delta \cdot e^{ \frac{\delta}{2 \cdot (1 - \delta)} \cdot \log(8/3)}\Big\}$, recalling that $\tilde{\eps} := \eps \cdot \frac{\delta}{1 - \delta} $.
\end{proof}

By scaling the quantities $\Delta\overline{\Phi}$ and $\Delta\overline{\Psi}$ in \cref{thm:main_ptw} by some $R > 0$ (usually the number of steps in each round, e.g., $R := b$ for the \Batched setting) and selecting a sufficiently small smoothing parameter $\gamma$, we obtain the main theorem.

{\renewcommand{\thethm}{\ref{thm:hyperbolic_cosine_expectation}}
	\begin{thm}[Restated]
\MainHyperbolicCosineExpectation
	\end{thm} }
	\addtocounter{thm}{-1}

\begin{proof}
Consider a labeling of the bins so that they are sorted in non-increasing order of the loads at step $t$. Applying \cref{thm:main_ptw} for the current load vector $x^t$ and the quantities \[
\Delta\overline{\Phi} := \sum_{i = 1}^n \Phi_i^t \cdot \left(p_i - \frac{1}{n}\right) \cdot \gamma \quad \text{ and } \quad \Delta\overline{\Psi} := \sum_{i = 1}^n \Psi_i^t \cdot \left(\frac{1}{n} - p_i\right) \cdot \gamma,
\]
we get that
\begin{align} \label{eq:application_res}
\Delta\overline{\Phi} + \Delta\overline{\Psi} %
\leq -\frac{\gamma\eps\delta}{4n} \cdot \Gamma^t + c\gamma\eps.
\end{align}
By the assumptions, 
\begin{align} \label{eq:assumption}
\Ex{\left. \Delta\Gamma^{t} \,\right|\, \mathfrak{F}^t} = \Ex{\left. \Delta\Phi^{t} \,\right|\, \mathfrak{F}^t} + \Ex{\left. \Delta\Psi^{t} \,\right|\, \mathfrak{F}^t} \leq R \cdot \left(\Delta\overline{\Phi} + \Delta\overline{\Psi} + K \cdot \frac{\gamma^2}{n} \cdot \Gamma^t \right).
\end{align}
Hence, combining \cref{eq:application_res} and  \cref{eq:assumption}, we get
\begin{align*}
\Ex{\left. \Delta\Gamma^{t} \right| \mathfrak{F}^t} 
& \leq R \cdot \left(- \frac{\gamma\eps\delta}{4n} \cdot \Gamma^t + c\gamma\eps + K \cdot \frac{\gamma^2}{n} \cdot \Gamma^t \right) 
  \leq - R \cdot \frac{\gamma\eps\delta}{8n} \cdot \Gamma^t + R \cdot c\gamma\eps,
\end{align*}
using that $\gamma \leq \frac{\eps\delta}{8K}$.

Finally, by \cref{lem:geometric_arithmetic}~$(ii)$, the second statement follows.
\end{proof}

\section{Tools for Verifying Drift Preconditions} \label{sec:three_useful_lemmas}

In this section, we prove three useful lemmas, for verifying the drift inequalities for the overload and underload potentials, which are the preconditions of \cref{thm:hyperbolic_cosine_expectation}.

We start by verifying the preconditions of \cref{thm:hyperbolic_cosine_expectation} for the sequential setting with weights, i.e., when we allocate to a single bin in each round.

\begin{lem} \label{lem:single_step_change}
Consider any allocation process with probability allocation vector $p^t$ satisfying condition \CTwo for $C > 1$ at every step $t \geq 0$. Further, consider the \Weighted setting with weights from a $\FiniteMgf(S)$ distribution with $S \geq 1$. Then, for the potentials $\Phi := \Phi(\gamma)$ and $\Psi := \Psi(\gamma)$ with any smoothing parameter $\gamma \in \big(0, \frac{1}{S} \big]$, we have for any step $t \geq 0$,
\begin{align} \label{eq:one}
\Ex{\left. \Delta\Phi^{t+1} \,\right|\, \mathfrak{F}^t} 
  \leq \sum_{i = 1}^n \Phi_i^t \cdot \left( \left( p_i^t - \frac{1}{n} \right) \cdot \gamma + 2CS \cdot \frac{\gamma^2}{n} \right),
\end{align}
and
\begin{align} \label{eq:two}
\Ex{\left. \Delta\Psi^{t+1} \,\right|\, \mathfrak{F}^t} 
  \leq \sum_{i = 1}^n \Psi_i^t \cdot \left( \left( \frac{1}{n} - p_i^t \right) \cdot \gamma + 2CS \cdot \frac{\gamma^2}{n} \right).
\end{align}
\end{lem}
\begin{proof}
Consider an arbitrary bin $i \in [n]$. Then, for the overload potential we have that
\begin{align*}
\Ex{\left. \Phi_i^{t+1} \,\right|\, \mathfrak{F}^t} 
 & = \Phi_i^t \cdot \Ex{\left. e^{\gamma \mathcal{W} (1 - 1/n)} \,\right|\, \mathfrak{F}^t} \cdot p_i^t + \Phi_i^t \cdot \Ex{\left. e^{-\gamma\mathcal{W}/n} \,\right|\, \mathfrak{F}^t} \cdot (1 - p_i^t) \\
 & \stackrel{(a)}{\leq} \Phi_i^t \cdot \left( 1 + \left(1 - \frac{1}{n} \right) \cdot \gamma + S\gamma^2 \right) \cdot p_i^t
 + \Phi_i^t \cdot \left( 1 - \frac{\gamma}{n} + S\frac{\gamma^2}{n^2} \right) \cdot (1 - p_i^t) \\
 & = \Phi_i^t \cdot \left( 1 + \left(p_i^t - \frac{1}{n} \right) \cdot \gamma + S \gamma^2 \cdot p_i^t + S \frac{\gamma^2}{n^2} \cdot (1 - p_i^t) \right) \\
 & \stackrel{(b)}{\leq} \Phi_i^t \cdot \left( 1 + \left(p_i^t - \frac{1}{n} \right) \cdot \gamma + 2CS \cdot \frac{\gamma^2}{n} \right),
\end{align*}
using in $(a)$ \cref{lem:bounded_weight_moment} twice with $\kappa := 1 - \frac{1}{n}$ and with $\kappa := -\frac{1}{n}$ respectively (and that $(1 - 1/n)^2 \leq 1$), and in $(b)$ that $p_i^t \leq \frac{C}{n}$ with $C > 1$ by condition \CTwo. Similarly, for the underloaded potential we have that
\begin{align*}
\Ex{\left. \Psi_i^{t+1} \,\right|\, \mathfrak{F}^t} 
 & = \Psi_i^t \cdot \Ex{\left. e^{-\gamma \mathcal{W} (1 - 1/n)} \,\right|\, \mathfrak{F}^t} \cdot p_i^t + \Psi_i^t \cdot \Ex{\left. e^{\gamma \mathcal{W}/n} \,\right|\, \mathfrak{F}^t} \cdot (1 - p_i^t) \\
 & \stackrel{(a)}{\leq} \Psi_i^t \cdot \left( 1 - \left(1 - \frac{1}{n} \right) \cdot \gamma + S\gamma^2 \right) \cdot p_i^t
 + \Psi_i^t \cdot \left( 1 + \frac{\gamma}{n} + S\frac{\gamma^2}{n^2} \right) \cdot (1 - p_i^t) \\
 & = \Psi_i^t \cdot \left( 1 + \left( \frac{1}{n} - p_i^t \right) \cdot \gamma + S \gamma^2 \cdot p_i^t + S \frac{\gamma^2}{n^2} \cdot (1 - p_i^t) \right) \\
 & \stackrel{(b)}{\leq} \Psi_i^t \cdot \left( 1 + \left(\frac{1}{n} - p_i^t \right) \cdot \gamma + 2CS \cdot \frac{\gamma^2}{n} \right),
\end{align*}
using in $(a)$ \cref{lem:bounded_weight_moment} twice with $\kappa := -(1 - \frac{1}{n})$ and with $\kappa := \frac{1}{n}$ respectively (and that $(1 - 1/n)^2 \leq 1$) and in $(b)$ that $p_i^t \leq \frac{C}{n}$ with $C > 1$ by condition \CTwo. This completes the proof.
\end{proof}

Combining \cref{lem:single_step_change} with \cref{thm:hyperbolic_cosine_expectation}, we obtain:
\begin{cor} \label{cor:like_ptw}
\CorLikePTW
\end{cor}
\begin{proof}
The result follows by \cref{thm:hyperbolic_cosine_expectation}, the preconditions of which are satisfied by \cref{lem:single_step_change} with $K := 2CS$ and $R := 1$, for $\Gamma := \Gamma(\gamma)$ with any smoothing parameter $\gamma \leq \frac{\eps\delta}{8K} = \frac{\eps\delta}{16CS}$.
\end{proof}

\begin{rem} \label{rem:tie_breaking}
The same upper bound in \cref{cor:like_ptw} also holds for allocation processes with a probability allocation vector $p^t$ satisfying the preconditions of \cref{lem:single_step_change} and using random tie-breaking. The reason for this is that $(i)$ averaging probabilities in \cref{eq:averaging_pi} can only reduce the maximum entry in the allocation vector $\tilde{q}^t$, i.e. $\max_{i \in [n]} \tilde{q}_i^t(x^t) \leq \max_{i \in [n]} q_i^t$, so it still satisfies $\mathcal{C}_2$ and $(ii)$~moving probability between bins $i$, $j$ with $x_i^t = x_j^t$ (and thus $\Phi_i^t = \Phi_j^t$ and $\Psi_i^t = \Psi_j^t$), implies that the aggregate upper bounds in \eqref{eq:one} and \eqref{eq:two} remain the same. %
\end{rem}

We now proceed with two lemmas that are useful for processes that allocate to more than one bin per round.

\begin{lem} \label{lem:z_i_t_bounded}
Consider any allocation process and let $Z_i^{t+1}:= y_i^{t+1}-y_i^{t}$ be the change of the normalized load of bin $i \in [n]$ in round $t$. Further, assume that for the allocation process there exists some $\gamma \in (0, 1]$ such that the following holds deterministically \[\gamma \cdot \sup_{t \geq 0} \max_{i\in [n]} |Z_i^{t+1}| \leq 1.\] Then, the overload potential $\Phi := \Phi(\gamma)$ satisfies for any bin $i \in [n]$
\[
\Ex{\left. \Phi_i^{t+1} \,\right|\, \mathfrak{F}^t}
 \leq \Phi_i^t \cdot \left( 1 + \Ex{Z_i^{t+1}} \cdot \gamma + \Ex{(Z_i^{t+1})^2} \cdot \gamma^2 \right),
\]
and the underloaded potential $\Psi := \Psi(\gamma)$ satisfies for any bin $i \in [n]$
\[
\Ex{\left. \Psi_i^{t+1} \,\right|\, \mathfrak{F}^t}
 \leq \Psi_i^t \cdot \left( 1 - \Ex{Z_i^{t+1}} \cdot \gamma + \Ex{(Z_i^{t+1})^2} \cdot \gamma^2 \right).
\]
\end{lem}

We start by proving a more general version where the uniform bound on $Z$ is replaced with a condition on the MGF of $Z$.

\begin{lem} \label{lem:z_i_t_mgf}
Consider any allocation process and let $Z_i^{t+1}$ be the change of the normalized load of bin $i \in [n]$ in round $t+1$. Further, assume that there exists $\gamma \in (0, 1]$, $\ell_1, \ell_2 > 0$ such that for any round $t \geq 0$,
\[
\Ex{e^{\pm \gamma Z_i^{t+1}}} \leq 1 \pm  \ell_1 \cdot \gamma + \ell_2 \cdot \gamma^2.
\]
Then, the overload potential $\Phi := \Phi(\gamma)$ satisfies for any bin $i \in [n]$
\[
\Ex{\left. \Phi_i^{t+1} \,\right|\, \mathfrak{F}^t}
 \leq \Phi_i^t \cdot \left( 1 + \ell_1 \cdot \gamma + \ell_2 \cdot \gamma^2 \right),
\]
and the underloaded potential $\Psi := \Psi(\gamma)$ satisfies for any bin $i \in [n]$
\[
\Ex{\left. \Psi_i^{t+1} \,\right|\, \mathfrak{F}^t}
 \leq \Psi_i^t \cdot \left( 1 - \ell_1 \cdot \gamma + \ell_2 \cdot \gamma^2 \right).
\]
\end{lem}
\begin{proof}[Proof of \cref{lem:z_i_t_mgf}]
We start with upper bounding the expected value of the overload potential $\Phi := \Phi(\gamma)$ over one round for any bin $i \in [n]$,
\begin{align*}
\Ex{\left. \Phi_i^{t+1} \,\right|\, \mathfrak{F}^t}
  & = \Ex{e^{\gamma (y_i^t + Z_i^{t+1})}} = e^{\gamma y_i^t} \cdot \Ex{e^{\gamma Z_i^{t+1}}} = \Phi_i^t \cdot \Ex{e^{\gamma Z_i^{t+1}}} \\
   & \leq \Phi_i^t \cdot \left(1 + \ell_1 \cdot \gamma + \ell_2 \cdot \gamma^2 \right).
\end{align*}
Similarly, for the underloaded potential $\Psi := \Psi(\gamma)$
\begin{align*}
\Ex{\left. \Psi_i^{t+1} \,\right|\, \mathfrak{F}^t}
  & = \Ex{e^{\gamma (-y_i^t - Z_i^{t+1})}} = e^{-\gamma y_i^t} \cdot \Ex{e^{-\gamma Z_i^{t+1}}} = \Psi_i^t \cdot \Ex{e^{-\gamma Z_i^{t+1}}} \\
  & \leq \Psi_i^t \cdot \left(1 - \ell_1 \cdot \gamma + \ell_2 \cdot \gamma^2 \right).
\end{align*}
This concludes the claim.
\end{proof}

We now return to the proof of \cref{lem:z_i_t_bounded}.
\begin{proof}[Proof of \cref{lem:z_i_t_bounded}]
Since $\gamma \cdot \max_{t \geq 0} \max_{i\in [n]} |Z_i^{t+1}| \leq 1$, using the Taylor estimate $e^v \leq 1 + v + v^2$ for any $|v| \leq 1$, we have that 
\[
  \Ex{e^{\gamma Z_i^{t+1}}} 
    \leq \Ex{1 + \gamma Z_i^{t+1} + (\gamma Z_i^{t+1})^2}
    = 1 + \Ex{Z_i^{t+1}} \cdot \gamma + \Ex{(Z_i^{t+1})^2} \cdot \gamma^2,
\]
and 
\[
  \Ex{e^{-\gamma Z_i^{t+1}}} 
    \leq \Ex{1 - \gamma Z_i^{t+1} + (\gamma Z_i^{t+1})^2}
    = 1 - \Ex{Z_i^{t+1}} \cdot \gamma + \Ex{(Z_i^{t+1})^2} \cdot \gamma^2.
\]
Finally, by \cref{lem:z_i_t_mgf} with $\ell_1 := \Ex{Z_i^{t+1}}$ and $\ell_2 := \Ex{(Z_i^{t+1})^2}$, we get the conclusion.
\end{proof}

\section{Applications} \label{sec:applications}

In this section, we prove an upper bound on the difference between the maximum and minimum load for a variety of processes. We primarily make use of \cref{thm:hyperbolic_cosine_expectation}, verifying its preconditions using the three auxiliary lemmas of \cref{sec:three_useful_lemmas}.

\subsection{The \texorpdfstring{\OnePlusBeta}{(1+β)}-process}
\label{sec:refined_expectation_one_plus_beta}
In this section, we improve the upper bound on the gap for the \OnePlusBeta-process from previous work for very small $\beta$ in the unit weights setting. In~\cite[Corollary 2.12]{PTW15}, it was shown that this gap is $\Oh(\log n/\beta + \log(1/\beta)/\beta)$. For $\beta = n^{-\omega(1)}$, the second term dominates. We improve this gap bound to $\Oh(\log n/\beta)$, which by \cite[Section 4]{PTW15} is tight for any $\beta$ being bounded away from $1$ (\cref{rem:one_plus_beta_lower_bound}).

\begin{thm}\label{thm:one_plus_beta_upper_bound}
Consider the \OnePlusBeta-process for any $\beta \in (0,1]$ in the \Weighted setting with weights from a $\FiniteMgf(S)$ distribution with $S \geq 1$. Then, there exists a constant $\kappa > 0$, such that for any step $m \geq 0$,
\[
\Pro{\max_{i \in [n]} \left| y_i^m \right| \leq \kappa \cdot S \cdot \frac{\log n}{\beta} } \geq 1 - n^{-2}.
\]
\end{thm}

\begin{proof}
By \cref{pro:one_plus_beta}, the \OnePlusBeta-process satisfies conditions \COne for $\eps := \frac{\beta}{2}$ and $\delta := \frac{1}{4}$ and \CTwo for $C := 2$. Hence, by \cref{cor:like_ptw}, there exists a constant $c := c(\delta) > 0$, such that for $\Gamma := \Gamma(\gamma)$ with $\gamma := \frac{\eps\delta}{16CS} = \frac{\beta}{256S}$, for any step $m \geq 0$,
\[
\Ex{\Gamma^m} \leq \frac{8c}{\delta} \cdot n.
\]
Hence, using Markov's inequality
\[
\Pro{\Gamma^m \leq \frac{8c}{\delta} \cdot n^3} \geq 1 - n^{-2}.
\]
When the event $\big\{ \Gamma^m \leq \frac{8c}{\delta} \cdot n^3 \big\}$ holds, we deduce the desired bound on the gap \[
\Gap(m) \leq \frac{1}{\gamma} \cdot \left( \log \left( \frac{8c}{\delta} \right) + 3 \cdot \log n \right) \leq \frac{4}{\gamma} \cdot \log n =  4 \cdot \frac{256S}{\beta} \cdot \log n = \Oh\left(S \cdot \frac{\log n}{\beta}\right). \qedhere
\]
\end{proof}

\begin{rem} \label{rem:one_plus_beta_lower_bound}
For the unit weights setting, this upper bound is tight up to multiplicative constants for any $\beta \leq 1 - \xi$ (for any constant $\xi \in (0, 1)$), due to a lower bound of $\Omega( \log n/\beta )$ shown in~\cite[Section~4]{PTW15} which holds with probability at least $1 - o(1)$.
\end{rem}

\begin{rem}
More generally, the fact that $\Gamma := \Gamma(\gamma)$ with $\gamma = \Theta(\beta)$ is $\Oh(n)$ in expectation, allows us to immediately deduce that with constant probability the number of bins with normalized load at least $+ z$ (or load at most $-z$) is $\Oh(n \cdot e^{-\Omega(\beta z)})$. These bounds are tighter than the ones attainable using the drift theorem in \cite{PTW15} and by \cref{lem:lower_bound_overload_height,lem:lower_bound_underload_height}, there is also a matching lower bound of $\Omega(n \cdot e^{-\Oh(\beta z)})$ for sufficiently large $m$.
\end{rem}

\subsection{Graphical Balanced Allocation with Weights} \label{sec:weighted_graphical}

In~\cite{PTW15}, the authors proved bounds on the gap for the \OnePlusBeta-process in the sequential setting where balls are sampled from a $\FiniteMgf(\zeta)$ distribution with constant $\zeta > 0$. Then, they used a majorization argument to deduce gap bounds for \TwoChoice in the \Graphical setting. However, due to the involved majorization argument not working for weights, all results for graphical allocations in~\cite{PTW15} assume balls have unit weights. This lack of results for weighted graphical allocations is summarized as~\cite[Open Question 1]{PTW15}:
\begin{quote}
\textit{Graphical processes in the weighted case. The analysis of section 3.1 goes through the majorization approach and therefore applies only to the unweighted case. It would be interesting to analyze such processes in the weighted case as well.}
\end{quote}
By leveraging the results in previous sections, we are able to address this open question.

For a $d$-regular (and connected) graph $G$, let us define the \textit{conductance} as:
\[
 \phi(G) := \min_{S \subseteq V \colon 1 \leq |S| \leq n/2} \frac{|E(S,V \setminus S)|}{|S| \cdot d},
\]
where $|E(X,Y)|$ counts (once) the edges between any two subsets $X$ and $Y$.

\begin{lem}\label{lem:expansion}
Consider \Graphical on a $d$-regular graph with conductance $\phi$. Then, in any step $t \geq 0$, the probability allocation vector $p^t$  satisfies for all $1 \leq k \leq n/2$,
\[
 \sum_{i=1}^k p_i^t \leq (1-\phi) \cdot \frac{k}{n},
\]
and similarly, for all $n/2+1 \leq k \leq n$,
\[
 \sum_{i=k}^{n} p_i^t \geq (1+\phi) \cdot \frac{n-k+1}{n}.
\]
Further, $\max_{i \in [n]} p_i^t \leq  \frac{2}{n}$. Thus, the vector $p^t$ satisfies condition \COne with $\delta := 1/2$, $\epsilon := \phi$ and condition \CTwo with $C := 2$.
\end{lem}
The proof of this lemma closely follows~\cite[Proof of Theorem~3.2]{PTW15}.

\begin{proof}
Fix any load vector $x^t$ in step $t$. Consider any $1 \leq k \leq n/2$. Let $S_k$ be the set of the $k$ bins with the largest load. Hence in order to allocate a ball into $S_k$, both endpoints of the sampled edge must be in $S_k$, and so
\begin{align*}
 \sum_{i=1}^k p_i^t &= \frac{2 \cdot |E(S_k,S_k)|}{2 \cdot |E|} 
 \leq \frac{d \cdot k  - \phi \cdot d \cdot k }{d \cdot n}
 = (1-\phi) \cdot \frac{k}{n},
\end{align*}
where the inequality used that $d \cdot |S_k| = |E(S_k, V \setminus S_k)| + 2 \cdot |E(S_k,S_k)|$ and the definition of the conductance $\phi$.
Now, we will consider the suffix sums for $n/2+1 \leq k \leq n$. We start by upper bounding the prefix sum up to $k-1$,
\begin{align*}
\sum_{i=1}^{k-1} p_i^t %
&\leq \frac{d \cdot (k-1) - \phi \cdot d \cdot (n - k + 1)}{d \cdot n} = \frac{(k-1) - \phi \cdot (n - k + 1)}{n},
\end{align*}
where we proceeded as in the bound above and used the definition of $\varphi$.
Therefore, we lower bound the suffix sum by
\[
\sum_{i = k}^n p_i^t = 1 - \sum_{i=1}^{k-1} p_i^t \geq 1 - \frac{(k-1) - \phi \cdot (n - k + 1)}{n} = (1+\phi) \cdot \frac{n - k + 1}{n}.
\]
Hence, $p^t$ satisfies condition \COne with $\epsilon := \phi$. 
Finally, we also know that $
p_i^{t} \leq \frac{d}{d \cdot n/2} = \frac{2}{n},
$
for any bin $i \in [n]$, since in the worst-case we allocate a ball to bin $i$ whenever one of its $d$ incident edges are chosen.
\end{proof}

Next we state our result for the \Graphical setting with weights.
\begin{thm}\label{thm:weighted_graphical}
Consider \Graphical on a $d$-regular graph with conductance $\phi > 0$. Further, assume that weights are sampled from a $\FiniteMgf(S)$ distribution with $S \geq 1$. Then, there exists a constant $\kappa > 0$ such that for any step $m \geq 0$, 
\[
\Pro{\max_{i \in [n]} |y_i^m| \leq \kappa \cdot S \cdot \frac{\log n}{\phi} } \geq 1-n^{-2}.
\]
\end{thm}
Note that if $G$ is a (family) of $d$-regular expanders, then $\phi$ is bounded below by a constant $>0$, and therefore we obtain a gap bound of $\Oh(\log n)$. Further, if $G$ is a complete graph, then \cref{thm:one_plus_beta_upper_bound} implies a gap bound of $\Oh(\log n)$. Both gap bounds are asymptotically tight for many weight distributions including the exponential distribution, see \cref{rem:weighted_lower_bound}.
\begin{proof}
By \cref{lem:expansion}, $p^t$ satisfies \COne  with $\eps := \phi$, $\delta := \frac{1}{2}$ and \CTwo with $C := 2$ in every step $t \geq 0$. So, by \cref{cor:like_ptw}, we have that for the for the potential $\Gamma := \Gamma(\gamma)$ with $\gamma := \frac{\eps \delta}{16CS} = \frac{\phi}{64S}$ and any step $m \geq 0$, 
\[
\Ex{\Gamma^m} \leq \frac{8c}{\delta} \cdot n,
\]
for some constant $c := c(\delta) > 0$. Hence, by Markov's inequality
\[
\Pro{\Gamma^m \leq \frac{8c}{\delta} \cdot n^3} \geq 1 - n^{-2}.
\]
The event $\{ \Gamma^m \leq \frac{8c}{\delta} \cdot n^3 \}$ implies that
\[
\max_{i \in [n]} |y_i^m| \leq \log \left(\frac{8c}{\delta} \right) + 3 \cdot \frac{64S}{\phi} \cdot \log n = \Oh\left( S \cdot \frac{\log n}{\phi} \right). \qedhere
\]
\end{proof}

\subsection{The \texorpdfstring{$\Quantile(\delta)$}{Quantile(δ)} process}

Now, we turn our attention to the $\Quantile(\delta)$ process. Recall that the $\Quantile(\delta)$ process has a time invariant probability allocation vector $p \in \R^n$ given by
\[
p_i = \begin{cases}
\frac{\delta}{n} & \text{for }i \leq n \cdot \delta, \\
\frac{1 + \delta}{n} & \text{otherwise}.
\end{cases}
\]

\begin{thm} \label{thm:quantile_upper_bound}
Consider the $\Quantile(\delta)$ process with any quantile $\delta \in (0, 1)$ in the \Weighted setting with weights from a $\FiniteMgf(S)$ distribution with $S \geq 1$. If $\delta \leq \frac{1}{2}$, there exists a constant $\kappa > 0$, such that for any step $m \geq 0$,\[
\Pro{\max_{i \in [n]} \left| y_i^m \right| \leq \kappa \cdot S \cdot \frac{\log n}{\delta}} \geq 1 - n^{-2}.
\]
If $\delta > \frac{1}{2}$, there exists a constant $\kappa > 0$, such that for any step $m \geq 0$,\[
\Pro{\max_{i \in [n]} \left| y_i^m \right| \leq \kappa \cdot S \cdot \frac{\log n}{1 - \delta}} \geq 1 - n^{-2}.
\]
\end{thm}
\begin{proof}
For any constant quantile $\delta \in (0, 1)$, the conclusion follows by \cref{cor:like_ptw}, using that the probability allocation vector of $\Quantile(\delta)$ satisfies conditions \COne and \CTwo (\cref{pro:quantile_satisfies_c1_and_c2}). 

For non-constant $\delta$, we cannot apply \cref{cor:like_ptw}. So, we consider two cases:

\textbf{Case 1 [$\delta \leq \frac{1}{3}$]:} We now show that the probability allocation vector of the $\Quantile(\delta)$ process satisfies condition \COne with $\tilde{\delta} := \frac{1}{3}$ and $\eps := 2\delta$, since
\[
\sum_{i = 1}^{n/3} p_i = \delta^2 + \left(\frac{n}{3} - n \cdot \delta\right) \cdot \frac{1 + \delta}{n} = \frac{1 - 2\delta}{3},
\]
and using that $p_i$ is non-decreasing in $i \in [n]$. Therefore, since $\max_{i \in [n]} p_i \leq \frac{2}{n}$ it satisfies conditions \COne and \CTwo and by \cref{cor:like_ptw}, we get the conclusion.

\textbf{Case 2 [$\delta > \frac{2}{3}$]:} In this case, we have that 
\[
p_{n/3} = \frac{\delta}{n} = \frac{1 - (1 - \delta)}{n}.
\]
Therefore, it satisfies condition \DOne at $\tilde{\delta} := \frac{1}{3}$ and $\eps := 1 - \delta$. Therefore, by \cref{pro:d0_and_d1_imply_c0_and_c1} it satisfies conditions \COne and \CTwo with the same parameters (and $C := 2$). So, by \cref{cor:like_ptw}, we get the conclusion.
\end{proof}

\subsection{The \texorpdfstring{$\TwinningWithQuantile(\delta)$}{\TwinningWithQuantile(δ)} Process} \label{sec:twinning_with_quantile}

Next, we analyze the \TwinningWithQuantile process defined in \cref{sec:notation_and_processes}.

\begin{thm} \label{thm:twinning_with_quantile_upper_bound}
Consider the $\TwinningWithQuantile(\delta)$ process for any constant $\delta \in (0, 1)$. Then, there exists a constant $\kappa := \kappa(\delta) > 0$, such that for any step $m \geq 0$,
\[
\Pro{\max_{i \in [n]} \left| y_i^m \right| \leq \kappa \cdot \log n} \geq 1 - n^{-2}.
\]
\end{thm}
\begin{proof}
Let $Z_i^{t+1}$ be the change of the normalized load of bin $i \in [n]$ in allocation $t+1$. We will verify the preconditions  of \cref{thm:hyperbolic_cosine_expectation}, by applying \cref{lem:z_i_t_bounded}. To this end, we will bound the first and second moment of $Z_i^{t+1}$.

We consider two cases based on the rank of a bin $i \in [n]$, splitting them into \textit{heavy} ($\Rank^t(i) \leq n \cdot \delta$) and \textit{light} ($\Rank^t(i) > n \cdot \delta$):

\textbf{Case 1 [$\Rank^t(i) \leq n \cdot \delta$]:} If we sample a heavy bin $i$, then we allocate one ball to it, so
\begin{align*}
\Ex{\left. Z_i^{t+1} \,\right|\, \mathfrak{F}^t}
 & = \underbrace{\left(1 - \frac{1}{n} \right) \cdot \frac{1}{n}}_{\text{Allocate one ball to bin }i}
 + \underbrace{\left( - \frac{1}{n} \right) \cdot \left(\delta - \frac{1}{n} \right)}_{\text{\makecell{Allocate one ball\\[-1.5pt] to any other heavy bin}}}
 + \underbrace{\left( - \frac{2}{n}\right) \cdot (1 - \delta)}_{\text{\makecell{Allocate two balls\\[-1.5pt] to any light bin}}}
 = \frac{\delta}{n} - \frac{1}{n}.
\end{align*}
Similarly, we bound the second moment
\begin{align*}
\Ex{\left. (Z_i^{t+1})^2 \,\right|\, \mathfrak{F}^t}
 & = \left(1 - \frac{1}{n} \right)^2 \cdot \frac{1}{n}
 + \left( - \frac{1}{n} \right)^2 \cdot \left(\delta - \frac{1}{n} \right)
 + \left( - \frac{2}{n}\right)^2 \cdot (1 - \delta) \leq \frac{2}{n}.
\end{align*}
\textbf{Case 2 [$\Rank^t(i) > n \cdot \delta$]:} If we sample a light bin $i$, then we allocate two balls to it, so
\begin{align*}
\Ex{\left. Z_i^{t+1} \,\right|\, \mathfrak{F}^t}
 & = \underbrace{\left(2 - \frac{2}{n} \right) \cdot \frac{1}{n}}_{\text{Allocate two balls to bin }i}
 + \underbrace{\left( - \frac{1}{n} \right) \cdot \delta}_{\text{\makecell{Allocate one ball \\[-1.5pt] to any heavy bin}}}
 + \underbrace{\left( -\frac{2}{n} \right) \cdot \left(1 - \delta - \frac{1}{n} \right)}_{\text{\makecell{Allocate two balls \\[-1.5pt]to any other light bin}}} = \frac{\delta}{n}.
\end{align*}
Similarly, we bound the second moment
\[
\Ex{\left. (Z_i^{t+1})^2 \,\right|\, \mathfrak{F}^t}
 = \left(2 - \frac{2}{n} \right)^2 \cdot \frac{1}{n}
 + \left( - \frac{1}{n} \right)^2 \cdot \delta
 + \left( -\frac{2}{n} \right)^2 \cdot \left(1 - \delta - \frac{1}{n} \right) \leq \frac{5}{n}.
\]
Therefore, by \cref{lem:z_i_t_bounded} with any $\gamma \leq \frac{1}{2}$ (since $\sup_{t\geq 0} \max_{i \in [n]} |Z_i^{t+1}| \leq 2$) we establish the preconditions of \cref{thm:hyperbolic_cosine_expectation} with  $\gamma := \frac{\eps \delta}{8K}$, $R:= 1$, $K := 5$ and probability vector
\[
p_i^t = \begin{cases}
\frac{\delta}{n} & \text{if }i \leq n \cdot \delta, \\
\frac{1+\delta}{n} & \text{otherwise},
\end{cases}
\]
which satisfies condition \COne with $\delta$ and $\eps := 1-\delta$ by \cref{pro:quantile_satisfies_c1_and_c2}, as it coincides with the probability allocation vector of $\Quantile(\delta)$. Therefore, applying \cref{thm:hyperbolic_cosine_expectation}, we get the high probability $\Oh(\log n)$ bound on the difference between maximum and minimum load.
\end{proof}

\subsection{The \texorpdfstring{$\QuantileWithPenalty(\delta)$}{\QuantileWithPenalty(δ)} Process} \label{sec:quantile_with_penalty}

We will now analyze the \QuantileWithPenalty process as defined in \cref{sec:quantile_with_penalty_def}.

\begin{thm} \label{thm:twinning_with_penalty_upper_bound}
Consider the $\QuantileWithPenalty(\delta)$ process for any constant $\delta \in (0, 1)$. Then, there exists a constant $\kappa := \kappa(\delta) > 0$, such that for any step $m \geq 0$, 
\[
\Pro{ \max_{i \in [n]} \left| y_i^m \right| \leq \kappa \cdot \log n} \geq 1 - n^{-2}.
\]
\end{thm}
By \cref{cor:twinning_with_quantile_lower_bound}, it follows that this bound on the gap is asymptotically tight.

\begin{proof}
Let $Z_i^{t+1}$ be the change of the normalized load of bin $i \in [n]$ in allocation $t+1$. We will verify the preconditions  of \cref{thm:hyperbolic_cosine_expectation}, by applying \cref{lem:z_i_t_bounded}. To this end, we will bound the first and second moment of $Z_i^{t+1}$.

\textbf{Case 1 [$\Rank^t(i) \leq n \cdot \delta$]:} We allocate two balls to heavy bin $i$ if the first sample was light and the second sample was $i$, so
\begin{align*}
\Ex{\left. Z_i^{t+1} \,\right|\, \mathfrak{F}^t} = \underbrace{\left( 2 - \frac{2}{n}\right) \cdot \frac{\delta}{n}}_{\text{Allocate two balls to bin }i}
 + \underbrace{\left( -\frac{1}{n} \right) \cdot \left(1 - \delta \right)}_{\text{\makecell{Allocate one ball \\ to any light bin}}} 
 + \underbrace{\left( -\frac{2}{n} \right) \cdot \delta \cdot \left(1 - \frac{1}{n} \right)}_{\text{\makecell{Allocate two balls to \\ any other second sample}}} = \frac{\delta}{n} - \frac{1}{n}.
\end{align*}
Similarly, we bound the second moment,
\[
\Ex{\left. (Z_i^{t+1})^2 \,\right|\, \mathfrak{F}^t} = \left( 2 - \frac{2}{n}\right)^2 \cdot \frac{\delta}{n}
 + \left( -\frac{1}{n} \right)^2 \cdot \left(1 - \delta \right) 
 + \left( -\frac{2}{n} \right) \cdot \delta \cdot \left(1 - \frac{1}{n} \right) \leq \frac{5}{n}.
\]

\textbf{Case 2 [$\Rank^t(i) > n \cdot \delta$]:} We allocate one ball to light bin $i$ if it is the first sample, and we allocate two balls to it if the first sample was heavy and the second sample was $i$, so
\begin{align*}
\Ex{\left. Z_i^{t+1} \,\right|\, \mathfrak{F}^t}
 & = \underbrace{\left(1 - \frac{1}{n} \right) \cdot \frac{1}{n}}_{\text{\makecell{Allocate one ball\\to bin $i$}}}
 + \underbrace{\left(2 - \frac{2}{n}\right) \cdot \frac{\delta}{n}}_{\text{\makecell{Allocate two balls\\to bin $i$}}}
 + \underbrace{\left( -\frac{1}{n} \right) \cdot \left(1 - \delta - \frac{1}{n} \right)}_{\text{\makecell{Allocate one ball to\\ any other light bin}}}  + \underbrace{\left( -\frac{2}{n} \right) \cdot \delta \cdot \left(1 - \frac{1}{n} \right)}_{\text{\makecell{Allocate two balls to\\ any other second sample}}}
 = \frac{\delta}{n}.
\end{align*}
Similarly, we bound the second moment,
\begin{align*}
\Ex{\left. (Z_i^{t+1})^2 \,\right|\, \mathfrak{F}^t} & =
\left(1 - \frac{1}{n} \right)^2 \cdot \frac{1}{n}
 + \left(2 - \frac{2}{n}\right)^2 \cdot \frac{\delta}{n}
 + \left( -\frac{1}{n} \right)^2 \cdot \left(1 - \delta - \frac{1}{n} \right) + \left( -\frac{2}{n} \right)^2 \cdot \delta \cdot \left(1 - \frac{1}{n} \right)\\ 
 & \leq \frac{5}{n}.
\end{align*}
Therefore, by \cref{lem:z_i_t_bounded} with any $\gamma \leq \frac{1}{2}$ (since $\sup_{t\geq 0} \max_{i \in [n]} |Z_i^{t+1}| \leq 2$) we establish the preconditions of \cref{thm:hyperbolic_cosine_expectation} with  $R:= 1$, $K := 5$ and probability vector 
\[
p_i^t = \begin{cases}
\frac{\delta}{n} & \text{if }i \leq n \cdot \delta, \\
\frac{1+\delta}{n} & \text{otherwise},
\end{cases}
\]
which satisfies condition \COne with $\delta$ and $\eps := 1-\delta$ by \cref{pro:quantile_satisfies_c1_and_c2}. Therefore, applying \cref{thm:hyperbolic_cosine_expectation}, we get the high probability $\Oh(\log n)$ bound on the difference between maximum and minimum load.
\end{proof}

\subsection{The \ResetMemory Process with Weights} \label{sec:reset_memory}

In this section, we analyze the \ResetMemory process. Interestingly, in the drift inequalities of $\Phi$ and $\Psi$, the probability allocation vector of the \TwoChoice process arises.

\begin{thm} \label{thm:reset_memory_upper_bound}
Consider the \ResetMemory process in the \Weighted setting with weights from a $\FiniteMgf(S)$ distribution for $S \geq 1$. Then, there exists a constant $\kappa := \kappa(S) > 0$ such that for any step $m \geq 0$,
\[
  \Pro{\max_{i \in [n]} \big| y_i^m\big| \leq \kappa \cdot S \cdot \log n} \geq 1 - n^{-2}.
\]
\end{thm}

\begin{proof}
Consider any even step $2t \geq 0$ for the \ResetMemory process and let $\mathcal{W}_1$ and $\mathcal{W}_2$ be the weights of the $(2t + 1)$-th and $(2t + 2)$-th balls. For the  $i$-th heaviest bin, we have for any $\gamma \leq 1/S$ that 
\begin{align*}
& \Ex{\left. e^{\pm \gamma Z_i^{2t+2}} \, \right|\, \mathfrak{F}^{2t}} \\
 & \qquad  = \underbrace{\Ex{e^{\pm\gamma (\mathcal{W}_1+\mathcal{W}_2) (1-1/n)}} \cdot \frac{1}{n} \cdot \frac{i}{n}}_{\text{Allocate both balls to bin $i$}}
 + \underbrace{\Ex{e^{\pm \gamma (\mathcal{W}_1 (1 - 1/n) - \mathcal{W}_2/n)}} \cdot \frac{1}{n} \cdot \left( 1 - \frac{i}{n}\right)  }_{\text{Allocate only the first ball to bin $i$} }  \\
 & \qquad \qquad + \underbrace{\Ex{e^{\pm\gamma (-\mathcal{W}_1/n + \mathcal{W}_2(1-1/n))}} \cdot \frac{i-1}{n} \cdot \frac{1}{n}}_{\text{Allocate only the second ball to bin $i$}}
 + \underbrace{\Ex{e^{\mp\gamma (\mathcal{W}_1 + \mathcal{W}_2)/n}} \cdot \left( 1 - \frac{1}{n} - \frac{i}{n^2} + \frac{1}{n^2} \right)}_{\text{Allocate $0$ balls to bin $i$}} \\
 & \stackrel{(a)}{\leq} 1 + \Bigg( \pm \left( 2 - \frac{2}{n} \right) \cdot \frac{1}{n} \cdot \frac{i}{n}  \pm \left( 1 - \frac{2}{n} \right) \cdot \frac{1}{n} \cdot \left(1 - \frac{i}{n} \right) \\
 & \qquad \qquad \pm \left( 1 - \frac{2}{n} \right) \cdot \frac{i-1}{n} \cdot \frac{1}{n} \mp \frac{2}{n} \cdot \left( 1 - \frac{1}{n} - \frac{i}{n^2} + \frac{1}{n^2} \right) \Bigg) \cdot \gamma + 6S \cdot \frac{\gamma^2}{n} \\
 & = 1 \pm \left( \frac{i-1 + n-i + 2i}{n^2} - \frac{2}{n} \right) \cdot \gamma + 6S \cdot \frac{\gamma^2}{n} \\
 & = 1 \pm \left( \frac{2i - 1}{n^2} - \frac{1}{n}\right) \cdot \gamma + 6S \cdot \frac{\gamma^2}{n},
\end{align*}
using in $(a)$ \cref{lem:bounded_weight_moment} with $S := S(\zeta) \geq \max\{1, 1/\zeta\}$ and that $\mathcal{W}_1$ and $\mathcal{W}_2$ are independent.

Therefore, by \cref{lem:z_i_t_mgf} with $\gamma := \frac{1}{384 S}$ we establish the preconditions of \cref{thm:hyperbolic_cosine_expectation} with  $R:= 1$, $K := 6S$ and probability vector $p_i^t := \frac{2i - 1}{n}$ which satisfies condition \COne with $\delta := \frac{1}{4}$ and $\eps := \frac{1}{2}$ by \cref{pro:d_choice_satisfies_c0_and_c1}. Therefore, applying \cref{thm:hyperbolic_cosine_expectation}, 
there exists a constant $c > 0$ such that for any step $2t \geq 0$,\[
\Ex{\Gamma^{2t}} \leq 32c \cdot n.
\]
By using Markov's inequality,
\[
\Pro{\Gap(2t) \leq \frac{1}{\gamma} \cdot \left( 3 \log n + \log(32c) \right)} \geq 1 - n^{-2}.
\]
Let $t := \lceil m/2 \rceil$. If $m = 2t$, then we are done, otherwise the $(2t+1)$-th ball \Whp~has weight $\Oh(S \cdot \log n)$, because of the precondition on the MGF, i.e., $\Ex{e^{\zeta W}}$ is constant. Hence, in that last step $\max_{i \in [n]} |y_i^m|$ cannot change by more than $\Oh(S \cdot \log n/n)$ and hence, we the claim follows. 
\end{proof}

\subsection{The \texorpdfstring{\Batched}{b-Batched} setting} \label{sec:b_batched_setting}

In this section, we consider the \Batched setting with unit weights for processes with a probability allocation vector satisfying conditions \COne and \CTwo with constant $C > 1$. For demonstration purposes, we only prove a slightly weaker result than the one in  \cite{LS22Batched}, in the sense that it is a factor of $\log n$ from the tight bound and we only consider the unit weight setting. 

\begin{thm}
Consider any allocation process with probability allocation vector $p^t$ satisfying condition \COne for constant $\delta \in (0, 1)$ and (not necessarily constant) $\eps \in (0,1)$ as well as condition \CTwo for some constant $C > 1$ at every step $t \geq 0$. Further, consider the \Batched setting with any $b \geq n$. Then, there exists a constant $\kappa := \kappa(\delta, C) > 0$, such that for any step $m \geq 0$ being a multiple of $b$,
\[
\Pro{\max_{i \in [n]} |y_i^m| \leq \kappa \cdot \frac{1}{\epsilon} \cdot \frac{b}{n} \cdot \log n } \geq 1 - n^{-2}.
\]
\end{thm}
\begin{proof}
Consider any $\gamma \in \big(0, \frac{n}{2Cb}\big]$. Then, 
\begin{align} \label{eq:batched_term_bound}
\left| \pm \left(p_i^t - \frac{1}{n} \right) \cdot b \cdot \gamma + p_i^t \cdot b \cdot \gamma^2 \right| 
 \stackrel{(a)}{\leq} \frac{C}{n} \cdot b \cdot \gamma + \frac{C}{n} \cdot b \cdot \gamma^2 
 \stackrel{(b)}{\leq} 2 \cdot \frac{C}{n} \cdot b \gamma
 \leq 1,
\end{align}
using in $(a)$ that $p_i^t \leq \frac{C}{n}$ (and $C > 1$) and in $(b)$ that $\gamma \leq 1$.

Let $Z_i^{t+b}$ be the change of the normalized load of bin $i \in [n]$ over one round (consisting of $b$ steps). We have that $Z_i^{t+b} = Y_i^{t+b} - \frac{b}{n}$ where $Y_i^{t+b} \sim \mathsf{Bin}(b, p_i^t)$. Hence, $\Ex{ Z_i^{t+b}} = (p_i^t - \frac{1}{n}) \cdot b$. In order to verify the precondition of \cref{lem:z_i_t_mgf}, we estimate 
\begin{align*}
\Ex{e^{\pm \gamma Z_i^{t+b}}}
 & \stackrel{(a)}{=} \Ex{e^{\pm \gamma Y_i^{t+b}}} \cdot e^{\mp \gamma b/n} \\
 & \leq (1 + p_i^t \cdot ( 1 - e^{\pm \gamma}))^b \cdot e^{\mp \gamma b/n} \\
 & \stackrel{(b)}{\leq} (1 + p_i^t \cdot ( \pm \gamma + \gamma^2))^b \cdot e^{\mp \gamma b/n} \\
 & \leq e^{\pm (p_i^t - 1/n) \cdot b  \cdot \gamma + p_i^t \cdot b \cdot \gamma^2} \\
 & \stackrel{(c)}{\leq} 1 \pm \left(p_i^t - \frac{1}{n}\right)  \cdot b \cdot \gamma + p_i^t \cdot b \cdot \gamma^2 + \left( \pm \left(p_i^t - \frac{1}{n} \right) \cdot b \cdot \gamma + p_i^t \cdot b \cdot \gamma^2 \right)^2 \\
 & \leq 1 \pm \left(p_i^t - \frac{1}{n}\right)  \cdot b \cdot \gamma + p_i^t \cdot b \cdot \gamma^2 + \left( 2 \cdot \frac{C}{n} \cdot b \cdot \gamma \right)^2 \\
 & \leq 1 \pm \left(p_i^t - \frac{1}{n}\right)  \cdot b \cdot \gamma + \frac{5C^2 b}{n} \cdot b \cdot \frac{\gamma^2}{n} \\
 &= 1 \pm \Ex{ Z_i^{t+b}   }  \cdot \gamma + \frac{5C^2 b}{n} \cdot b \cdot \frac{\gamma^2}{n},
\end{align*}
using in $(a)$ that $\ex{e^{tY}} = (1 - p + pe^t)^b$ for $Y \sim \mathsf{Bin}(b, p)$, in $(b)$ the Taylor estimate $e^v \leq 1 + v + v^2$ for $|v| \leq 1$ and $\gamma \leq 1$, in $(c)$ the same Taylor estimate and inequality \cref{eq:batched_term_bound}.

Therefore, by \cref{lem:z_i_t_mgf} with any $\gamma \leq \frac{1}{2}$, $\ell_1 := 1$ and $\ell_2 := \frac{5C^2b^2}{n^2}$ we establish the preconditions of \cref{thm:hyperbolic_cosine_expectation} with  $R:= b$, $K := \frac{5C^2b}{n}$ and probability vector $p^t$ which satisfies condition \COne and \CTwo by assumption. Therefore, applying \cref{thm:hyperbolic_cosine_expectation} with $\gamma := \frac{\eps\delta}{8K} = \frac{\eps \delta n}{40 C^2 b}$, we get \Whp~the $\Oh\big(\frac{1}{\eps} \cdot \frac{b}{n} \cdot \log n\big)$ bound on the difference between maximum and minimum load.
\end{proof}

\section{Lower Bounds} \label{sec:lower_bounds}

In this section, we prove a general lower bound that applies to any process that allocates at least one ball (of unit weight) with uniform probability in each round (which could consist of multiple steps). This implies an $\Omega(\log n)$ bound for the gap of  $\TwinningWithQuantile(\delta)$, $\QuantileWithPenalty(\delta)$ and \ResetMemory for any $\delta = \Theta(1)$.

In the following lemma, we prove a lower bound for general allocation processes which allocate a constant number of balls in each round and at least one ball with uniform probability across the bins.

\begin{lem} \label{lem:general_lower_bound}
Consider an allocation process
for which there are constants $c > 0$ and $c_{max} > 0$ such that in each round $t \geq 0$, the process allocates at least one ball with uniform probability $q \geq c/n$ over all bins, and at most $c_{\max}$ balls in total. Then, there exists a round $m := m(c, c_{\max})$ and constant $\kappa : = \kappa(c, c_{\max}) > 0$ such that
\[
  \Pro{\Gap(m) \geq \kappa \cdot \log n} \geq 1 - n^{-1}.
\]
\end{lem}
\begin{proof}
Let $m := C n \log n$, where $C := \frac{c}{400 \cdot e \cdot c_{\max}^2 }$. By a Chernoff bound (\cref{lem:multiplicative_factor_chernoff_for_binomial}), \Whp~in these $m$ steps at least $\frac{c}{e} \cdot m$ balls are being allocated using a \OneChoice process. 

By, e.g.~\cite[Section 4]{PTW15}, when $r n\log n$ balls are allocated using \OneChoice, for any constant $r > 0$, then with probability at least $1 - n^{-2}$, the maximum load is at least $(r+\frac{1}{10} \sqrt{r}) \log n$. Using this with $r:=\frac{c^2}{400 \cdot e^2 \cdot c_{\max}^2}$ and applying the union bound, we conclude that
\[
\Pro{\max_{i \in [n]} x_i^m \geq \left( \frac{c^2}{400   \cdot e^2 \cdot c_{\max}^2} + \frac{1}{10} \cdot \sqrt{\frac{c^2}{400   \cdot e^2 \cdot c_{\max}^2}} \right) \cdot \log n} \geq 1 - n^{-1}.
\]
Further, at step $m$ the average load can be bounded by
\[
\frac{W^m}{n} \leq \frac{m \cdot c_{\max}}{ n} = C \cdot c_{\max} \cdot \log n = \frac{c}{400 \cdot e \cdot c_{\max}} \cdot \log n = \frac{1}{20} \cdot \sqrt{ \frac{c^2}{400 \cdot e^2 \cdot c_{\max}^2}} \cdot \log n.
\]
Therefore,
\[
\Pro{\Gap(m) \geq \frac{c}{400 \cdot e \cdot c_{\max}} \cdot \log n} \geq 1 - n^{-1}. \qedhere
\]
\end{proof}

The $\TwinningWithQuantile(\delta)$ process has a uniform probability of $\delta/n$ to allocate (at least) one ball over all bins and allocates at most two balls per step, so by \cref{lem:general_lower_bound} we get a matching lower bound.  

\begin{cor} \label{cor:twinning_with_quantile_lower_bound}
Consider the $\TwinningWithQuantile(\delta)$ process for any constant quantile $\delta \in (0, 1)$. Then, there exist constants $c := c(\delta)$ and $\kappa := \kappa(\delta) > 0$, such that\[
  \Pro{\Gap(cn \log n) \geq \kappa \cdot \log n} \geq 1 - n^{-1}.
\]
\end{cor}

The $\QuantileWithPenalty(\delta)$ process has a uniform probability of $\delta/n$ to allocate one ball over all bins and allocates at most two balls per step, so by \cref{lem:general_lower_bound} we get a matching lower bound.  

\begin{cor} \label{cor:twinning_with_penalty_lower_bound}
Consider the $\QuantileWithPenalty(\delta)$ process for any constant quantile $\delta \in (0, 1)$. Then, there exist constants $c := c(\delta)$ and $\kappa := \kappa(\delta) > 0$, such that\[
  \Pro{\Gap(cn \log n) \geq \kappa \cdot \log n} \geq 1 - n^{-1}.
\]
\end{cor}

Since \ResetMemory allocates one ball using \OneChoice every two steps, by \cref{lem:general_lower_bound} considering rounds with two steps, we obtain a matching lower bound.

\begin{cor}\label{cor:reset_memory_lower_bound}
Consider the \ResetMemory process. Then, there exist constants $\kappa > 0$ and $c > 0$, such that 
\[
\Pro{\Gap(c n \log n) \geq \kappa \cdot \log n} \geq 1 - n^{-1}.
\]
\end{cor}

Further, for any process making a constant number of bin samples in each round, it follows by a coupon collector's argument, that one bin is not sampled in the first $m = \Theta(n \log n)$ rounds.
\begin{rem}
Consider any allocation process which in each round allocates to a subset of a constant number of $d$ bins sampled uniformly at random and allocates to at least one of them a unit weight ball. Then, there exists a constant $\kappa > 0$ and a constant $c$, such that,
\[
  \Pro{\min_{i \in [n]} y_i^{cn \log n}  \leq - \kappa \cdot \log n} \geq 1 - n^{-1}.
\]
\end{rem}

Finally, for any processes in the \Weighted setting with a suitable distribution, \Whp~at least one ``heavy'' ball will emerge in the first $n$ allocations (and also the average load only increases to $\Oh(1)$). Putting these observations together immediately yields:
\begin{rem} \label{rem:weighted_lower_bound}
Consider any allocation process in the \Weighted setting with weights from an $\mathsf{Exp}(1)$ distribution. Then,\[
\Pro{\Gap(n) \geq \frac{1}{2} \cdot \log n} \geq 1 - n^{-1}.
\]
\end{rem}

As described in \cref{sec:bounds_on_bins_with_load_at_least}, the fact that the hyperbolic cosine potential is $\Oh(n)$ in expectation gives an $\Oh(n \cdot e^{-\Omega(z)})$ bound on the number of bins with normalized load at least $+ z$. In the following, we show that these bounds are asymptotically tight for all values of $z \in (0, \Oh(\log n))$ for the \OnePlusBeta-process with constant $\beta \in (0, 1)$ and $\Quantile(\delta)$ with constant $\delta \in (0, 1)$.

\begin{lem}
\label{lem:lower_bound_overload_height}
Consider any allocation process
for which there are constants $c \in (0, 1]$ and $c_{max} \geq 1$ such that in each round $t \geq 0$, the process allocates at least one ball with uniform probability $q \geq c/n$ over all bins, and at most $c_{\max}$ balls in total. Then, there exist constants $c_1, c_2 > 0$ such that for any $z \in (0, c_1 \log n)$ and $m := n \cdot z$,
\[
  \Pro{\Bigl| i \in [n] : y_i^m \geq z \Bigr| \geq \frac{1}{2} n  \cdot e^{-c_2 z}} \geq 1 - \frac{1}{2} e^{-1/8}.
\]
\end{lem}
\begin{proof}
We define a coupling between the allocations of the process and a \OneChoice process. In each step $t \geq 0$:
\begin{itemize}
  \item With probability $c$, we allocate following the \OneChoice process (i.e., we allocate to a bin chosen uniformly at random).
  \item Otherwise, we allocate to a bin $i \in [n]$ chosen with probability $\frac{p_i^t - c/n}{1 - c}$.
\end{itemize}
We will show that just the \OneChoice allocations are enough to produce sufficiently many bins with load at least $c_{\max} \cdot \frac{m}{n} + z$ and this will imply the statement as the average can change by at most $c_{\max}/n$ in each step.

We will use the Poisson Approximation~\cite[Chapter 5]{MU17}. To this end, let $Z_1,Z_2,\ldots,Z_n$ be $n$ independent Poisson random variables with parameter $\lambda=\frac{m'}{n}=\frac{c}{2} \cdot \frac{m}{n} = \frac{c}{2} z$. We will now use the following lower bound for each Poisson random variable $Z_i \sim \mathsf{Poi}(\lambda)$,
\[
  \Pro{Z_i \geq \lambda + u} \geq \Pro{Z_i = \lambda + u} = \frac{e^{-\lambda} \cdot \lambda^{\lambda + u}}{(\lambda + u)!}
  \geq \frac{e^{-\lambda} \cdot \lambda^{\lambda + u}}{e \cdot (\lambda + u) \cdot \big(\frac{\lambda + u}{e}\big)^{\lambda + u}},
\]
using that $v! \leq e \cdot v \cdot \big(\frac{v}{e}\big)^v$. We choose $u$ such that
\[
\lambda + u = c_{\max} \cdot \frac{m}{n} + z \quad \Leftrightarrow \quad
u = \left(c_{\max} + 1 - \frac{c}{2}\right) \cdot z.
\]
Therefore, 
\begin{align*}
\Pro{Z_i \geq \lambda + u}
 & \geq \exp\left( - \lambda + (\lambda + u) \cdot \left( 1 +  \log \frac{\lambda}{\lambda + u} \right) - 1 - \log(\lambda + u) \right) \\
 & \stackrel{(a)}{\geq} \exp\left( - \lambda + (\lambda + u) \cdot \left(1 - \frac{u}{\lambda}\right) - 1 - \log(\lambda + u) \right) \\
 & = \exp\left(  - \frac{u^2}{\lambda} - 1 - \log(\lambda + u) \right) \\
 & \stackrel{(b)}{\geq} \exp( - c_2 z),
\end{align*}
where in $(a)$ we used the Taylor estimate $\log \frac{1}{1 + v} \geq - v$ (since $e^v \geq 1 + v$), and $(b)$ holds for some constant $c_2 > 0$, using that $z \geq c_3$ for some sufficiently large constant $c_3 > 0$. For $z \in (0, c_3)$, this follows trivially by considering the bins with normalized load at least $c_3$ (and a smaller $c_2$). 

Therefore, by applying a standard Chernoff bound with $\delta := \frac{1}{2}$, we have that
\[
  \Pro{\Bigl| i \in [n] : Z_i \geq c_{\max} \cdot \frac{m}{n} + z \Bigr| \geq \frac{1}{2} n  \cdot e^{-c_2 z}} \geq 1 - e^{-\frac{1}{8} \cdot n  \cdot e^{-c_2 z}} \geq 1 - o(1),
\]
using that $z < c_1 \log n$ for sufficiently small constant $c_1 > 0$. Let $M$ be the number of balls allocated with \OneChoice, then by the previous inequality we obtain that\[
\Pro{ \left. \Bigl| i \in [n] : y_i^m \geq z \Bigr| \geq \frac{1}{2} n  \cdot e^{-c_2 z} \,\right|\, M \geq m'} \geq 1 - o(1),
\]
By a Chernoff bound with $\delta := \frac{1}{2}$ we have that the number of balls $M$ allocated using \OneChoice is 
\[
  \Pro{M \geq m'} = \Pro{M \geq \left( 1 - \frac{1}{2} \right) \cdot c \cdot m} \geq 1 - e^{-1/8}.
\]
By combining the last two inequalities, we get the conclusion.
\end{proof}

Similarly, we prove tight bounds for the number of bins with an underload of at least $- z$ for $z \in (0, \Oh(\log n))$.

\begin{lem} \label{lem:lower_bound_underload_height}
Consider any allocation process which in each round allocates to a subset of a constant number of $d$ bins sampled uniformly at random and allocates on aggregate at least one unit weight. Then,  for any $z \in (0, \frac{1}{8d} \log n)$  and $m := n \cdot z$,
\[
  \Pro{\Bigl| i \in [n] : y_i^m \leq -z \Bigr| \geq \frac{1}{2} n  \cdot e^{-2dz}} \geq 1 - n^{-1}.
\]
\end{lem}
\begin{proof}
Consider any fixed bin $i \in [n]$. If this bin is not sampled in any of the first $m$ steps, then we have that 
\[
  y_i^m \leq - \frac{m}{n} = -z, 
\]
since in each step we are allocating at least one unit weight ball. This occurs with probability
\[
  p := \left( 1 - \frac{d}{n} \right)^m \geq \left( e^{-\frac{d}{n} -\frac{d^2}{n^2}} \right)^m
  \geq \left( e^{-2\frac{d}{n}} \right)^m
  = e^{-2d \cdot \frac{m}{n}},
\]
using that $1-x \geq e^{-x-x^2} $ for any $x \in (0,0.5)$.
Let $f := f(Z_{11}, \ldots, Z_{1d}, \ldots, Z_{m1}, \ldots , Z_{md})$ be the number of bins that were not sampled in the first $m$ steps. So,
\[
\Ex{f} = n \cdot p \geq n \cdot e^{-2d \cdot \frac{m}{n}}.
\]
Then, changing any one of the $md$ samples can change $f$ by at most one, so for any $(i,j) \in [m] \times [d]$,
\[
 \max_{z_{1,1}, \ldots, z_{m,d} \in [n]} %
\max_{z_{i,j}, z_{i,j}' \in [n]} \left| f(Z_{1,1}, \ldots, z_{i,j}, \ldots , Z_{m,d}) - f(Z_{1,1}, \ldots, z_{i,j}', \ldots , Z_{m,d}) \right| \leq 1.
\]
Therefore, applying the method of bounded differences (\cref{lem:mobd}) gives
\[
  \Pro{\Bigl| i \in [n] : y_i^m \leq -z \Bigr| \geq \frac{1}{2} n  \cdot e^{-2d z}} \geq \exp\left(- \frac{n^2 \cdot e^{-4d z}}{2n} \right) \geq 1 - n^{-1},
\]
using that $z < \frac{1}{8d} \log n$.
\end{proof}

\section{Conclusions} \label{sec:conclusion}

In this work we improved the drift theorem for balanced allocation processes from~\cite{PTW15}. This allowed us to obtain new bounds for a wider variety of processes and improve bounds on existing processes. 

We believe this drift theorem will be helpful in analyzing other allocation processes and possibly even more general settings that go beyond balls-and-bins. %
One specific open problem is to establish tight bounds for $\Quantile(\delta)$ with non-constant $\delta \in (0, 1)$. In particular, we think there exist choices of $\delta$ for which the gap would match the general lower bound of $\Omega(\log n/\log \log n)$, which holds for any \TwoThinning process. Finally, it is also open to establish tight bounds for the \KQuantile process~\cite{LS22Queries} with $k > 1$ and $\DThinning$ with $d > 2$. Another avenue for future work is to study allocation processes where balls can be removed or the weights of the balls are chosen by an adversary (instead of being sampled randomly).

\addcontentsline{toc}{section}{Bibliography}
\renewcommand{\bibsection}{\section*{Bibliography}}
\setlength{\bibsep}{0pt plus 0.3ex}
\bibliographystyle{ACM-Reference-Format-CAM}
\bibliography{bibliography}

\clearpage

\appendix

\section{Tools}

\subsection{Auxiliary Probabilistic Claims}

For convenience, we add the following well-known inequality for a sequence of random variables, whose expectations are related through a recurrence inequality.

\begin{lem} \label{lem:geometric_arithmetic}
Consider any sequence of random variables $(X^i)_{i \in \mathbb{N}}$ for which there exist $a \in (0, 1)$ and $b > 0$, such that every $i \geq 1$,
\[
\Ex{X^i \mid X^{i-1}} \leq X^{i-1} \cdot a + b.
\]
Then, if $X^0 \leq \frac{b}{1-a}$ holds, then for every $i \geq 0$,
\[
\Ex{X^i} \leq \frac{b}{1 - a}.
\]
\end{lem}
\begin{proof}
We will prove this claim by induction. Then, assuming that $\Ex{X^i} \leq \frac{b}{1-a}$ holds for $i \geq 0$, we have for $i+1$
\begin{align*}
\Ex{X^{i+1}}
  & = \Ex{\Ex{X^{i+1} \mid X^{i}}}  \leq \Ex{X^{i} } \cdot a + b \leq \frac{b}{1-a} \cdot a + b = \frac{b}{1-a}. \qedhere 
\end{align*}
\end{proof}

We give a proof for the well-known fact that when $\Ex{e^{\zeta \mathcal{W}}} =\Theta(1)$ then $\ex{\mathcal{W}^4}$ is also bounded. 

\begin{lem}\label{lem:s_bound}
Consider any non-negative random variable $\mathcal{W}$ with $\Ex{e^{\zeta \mathcal{W}}} < \infty$ for some $\zeta > 0$. then 
\[
\Ex{\mathcal{W}^4} < \left(\left(\frac{8}{\zeta}\right) \cdot \log\left(\frac{8}{\zeta}\right) \right)^4 +  \Ex{e^{\zeta \mathcal{W}}}.
\]
\end{lem}
\begin{proof}
Let $\kappa := (8/\zeta) \cdot \log(8/\zeta)$. Consider any $x \geq \max(0, \kappa) =: \kappa^*$. Then
\[
e^{\zeta x/4} = e^{\zeta x/8} \cdot e^{\zeta x/8} \geq e^{\log(8/\zeta)} \cdot e^{\zeta x/8} \geq \frac{8}{\zeta} \cdot \frac{\zeta x}{8} = x,
\]
using that $e^v \geq v$ for any $v$. Hence,
\[
e^{\zeta x} = (e^{\zeta x / 4})^4 \geq x^4.
\]
Hence, if $p_x$ is the pdf of $\mathcal{W}$, then
\begin{align*}
\Ex{\mathcal{W}^4} 
 & = \int_{x=0}^{\infty} x^4 \cdot p_x dx = \int_{x=0}^{\kappa^*} x^4 \cdot p_x dx + \int_{x=\kappa^*}^{\infty} x^4 \cdot p_x dx \\
 & \leq \int_{x=0}^{\kappa^*} \kappa^4 \cdot p_x dx + \int_{x=\kappa^*}^{\infty} e^{\zeta x} \cdot p_x dx \\
 & \leq \kappa^4 \cdot \int_{x=0}^{\infty} p_x dx + \int_{x=0}^{\infty} e^{\zeta x} \cdot p_x dx = \kappa^4 + \Ex{e^{\zeta \mathcal{W}}}. \qedhere
\end{align*}
\end{proof}

\subsection{Auxiliary Deterministic Inequalities}

\begin{lem}[{\cite[Lemma A.7]{LS22Batched}}]\label{lem:quasilem2}Let $(p_k)_{k=1}^n , (q_k)_{k=1}^n $ be two probability vectors and $(c_k)_{k=1}^n$ be non-negative and non-increasing. Then if $p$ majorizes $q$, then \begin{equation*}
\sum_{k = 1}^n p_k \cdot c_k \geq \sum_{k = 1}^n q_k \cdot c_k.
\end{equation*}
\end{lem}

\begin{lem} \label{lem:decreasing_fn}
The function $f(z) = z \cdot e^{k/z}$ for any $k > 0$, is decreasing for $z \in (0, k]$.
\end{lem}
\begin{proof}
By differentiating,
\[
f'(z) = e^{k/z} - z \cdot e^{k/z} \cdot \frac{k}{z^2} = e^{k/z} \cdot \Big( 1 - \frac{k}{z}\Big).
\]
For $z \in (0, k]$, $f'(z) \leq 0$, so $f$ is decreasing.
\end{proof}

\subsection{Concentration Inequalities}

We start with a well-known form of the Chernoff bound for binomial random variables.

\begin{lem}[Multiplicative Factor Chernoff Binomial Bound~\cite{MRS01}] %
\label{lem:multiplicative_factor_chernoff_for_binomial}
Let $X_1, \ldots , X_n$ be independent binary random variables with $\Pro{X^i = 1} = p$. Then,
\[
\Pro{\sum_{i = 1}^n X^i \leq \frac{np}{e}} 
\leq e^{\left(\frac{2}{e} - 1 \right)np}.
\]
\end{lem}
We proceed with the method of bounded independent differences.

\begin{lem}[Method of Bounded Independent Differences {\cite[Corollary 5.2]{DubPan}}]\label{lem:mobd} Let $f$ be a function of $N$ independent random variables $X^1 ,\dots , X^N$, where each $X^i$ takes values in a set $\Omega^i$. Assume that for each $i \in [N]$ there exists a $c_i \geq 0 $ such that
	\[ \left|f(x_1 ,\dots,x_{i-1}, x_{i},x_{i+1},\dots, x_N) - f(x_1 ,\dots,x_{i-1}, x_{i}',x_{i+1},\dots, x_N) \right| \leq  c_i,\] for any $x_1\in \Omega_1,\dots,x_{i-1}\in \Omega_{i-1}, x_{i},x_{i}'\in \Omega_{i} ,x_{i+1}\in \Omega_{i+1},\dots, x_N \in \Omega_N$. Then, for any $\lambda>0$,
	\[\Pro{f < \Ex{f} - \lambda }  \leq \exp\left(- \frac{\lambda^2}{2 \cdot \sum_{i=1}^N c_i^2} \right).\]\end{lem}

\section{Remark on Condition \texorpdfstring{$\mathcal{C}_2$}{C₂}} \label{sec:condition_c2}

\begin{clm} \label{clm:worst_case_process}
Consider a probability vector $p$ satisfying conditions \DZero and \DOne for some constant $\delta$ and any $\eps$. Then, there exists a probability vector $q$ satisfying \COne  with the same $\delta$ and $\eps$ and \CTwo for $C := C(\delta)$, such that $q \succeq p$.
\end{clm}
\begin{proof}
Such a probability vector (in fact, the worst-case vector) can be defined as \[
q_i := \begin{cases}
    \frac{1-\eps}{n} & \text{if }i \leq n \cdot \delta \\
    \frac{1+\frac{\delta \eps}{1-\delta}}{n} & \text{otherwise}.
\end{cases}
\]
This probability vector majorizes $p$ and satisfies condition \COne with $\delta$ and $\eps$, and also \CTwo with $C := 1+\frac{\delta \eps}{1-\delta}= \Theta(1)$.
\end{proof}

\begin{rem}
Consider any process satisfying conditions \DZero and \DOne in the unit weights case. Then, the analysis in \cite{PTW15} implies an $\Oh(\log n/\eps + \log(1/\eps)/\eps)$ gap. \cref{cor:like_ptw} implies an $\Oh(\log n/\eps)$ gap by majorization (\cref{thm:majorization}) with the worst-case process $\mathcal{Q}$ as defined in \cref{clm:worst_case_process} and using that this process satisfies condition \CTwo with $C := \Theta(1)$.
\end{rem}

\end{document}